\documentclass[a4paper]{amsart}
\usepackage{graphicx} 
\usepackage[dvipsnames]{xcolor}
\usepackage{setspace}
\usepackage[hang,flushmargin,symbol*]{footmisc}
\usepackage{amsmath}
\usepackage{amsthm}
\usepackage{amssymb}
\usepackage{mathtools}
\usepackage{enumitem}
\usepackage{calc}
\usepackage{caption}
\usepackage[labelformat=simple,labelfont={}]{subcaption}
\usepackage{tikz}
\usepackage{tikz-cd}
\usetikzlibrary{decorations.markings, calc}
\usetikzlibrary{arrows,shapes,positioning}
\usepackage{url}
\usepackage{array}
\usepackage{mathrsfs}

\usepackage{verbatim}

\usepackage{subfiles}

\usepackage[colorinlistoftodos, loadshadowlibrary]{todonotes}

\usepackage{dashbox}

\theoremstyle{definition}
\newtheorem{theorem}{Theorem}[section]
\newtheorem{lemma}[theorem]{Lemma}

\newtheorem{corollary}[theorem]{Corollary}
\newtheorem{proposition}[theorem]{Proposition}

\newtheorem{conjecture}[theorem]{Conjecture}
\newtheorem{definition}[theorem]{Definition}
\newtheorem{example}[theorem]{Example}
\newtheorem{remark}[theorem]{Remark}

\newtheorem{mainthm}{Theorem}

\usepackage{marginnote}

\usepackage{amsfonts}
\usepackage[colorlinks=true]{hyperref}
\usepackage{ mathdots }
\usepackage{dutchcal}
\usepackage{extarrows}

\usepackage{xparse}

\newcommand{\ch}{{\rm ch}}

\usepackage{multirow}

\usepackage[normalem]{ulem}

\renewcommand{\tilde}[1]{\ensuremath{\widetilde{#1}}}

\newcommand{\pt}{{\rm pt}}

\newcommand{\OO}{\mathcal{O}}
\newcommand{\sS}{\mathcal{S}}

\newcommand{\Aff}{{\mathbb{A}}}
\newcommand{\PP}{\mathbb{P}}
\newcommand{\GG}{\mathbb{G}}

\newcommand{\Hom}{\text{Hom}}

\newcommand{\Ext}{\text{Ext}}
\newcommand{\HH}{\text{HH}}
\newcommand{\RHom}{\text{RHom}}

\newcommand{\lkah}[1]{\Omega^{1,{\rm log}}_{#1}}

\newcommand{\id}{\rm id}

\newcommand{\Sym}{{\rm Sym}}

\newcommand{\HHl}[1]{{\rm HH}^\ell_{#1}}
\newcommand{\LogPair}{\mathbf{LogPair}}
\newcommand{\cHHl}[1]{{\rm HH}^{\ell #1}}

\usepackage{stmaryrd}

\definecolor{sebgreen1}{rgb}{0.019,0.317,0.149}
\definecolor{sebgreen2}{rgb}{0.784,0.952,0.780}

\newcommand{\af}[1]{\Theta_{#1}}

\newcommand{\bra}[1]{{\left[{#1}\right]}}

\title[Functoriality of logarithmic Hochschild homology]{Functoriality of logarithmic Hochschild homology of log smooth pairs}
\author{\'Ad\'am Gyenge}
\address{Department of Algebra and Geometry, Institute of Mathematics, Budapest University of Technology and Economics, 
M\H{u}egyetem rakpart 3, H-1111, 
Budapest, Hungary}
\email{Gyenge.Adam@ttk.bme.hu}

\author{M\'arton Hablicsek}
\address{Mathematics Institute, Leiden University, Niels Bohrweg 1, 2333 CA Leiden, Netherlands}
\email{hablicsekhm@math.leidenuniv.nl}

\author{Leo Herr}
\address{Department of Mathematics, Virginia Tech, 225 Stanger Street, Blacksburg, VA 24061-1026, USA}
\email{herr@vt.edu}

\date{}

\subjclass[2020]{Primary 14F08; Secondary 13D03}
\keywords{log schemes, Hochschild homology, Fourier-Mukai transform, Chern character}

\begin{document}

\maketitle

\begin{abstract}
The construction of a satisfactory dg category of logarithmic coherent sheaves remains a central open problem in logarithmic geometry. In this paper, we propose an alternative correspondence-theoretic approach based on logarithmic Fourier--Mukai transforms.
For smooth proper log pairs, we introduce strong log Fourier--Mukai kernels supported on canonical blow-up compactifications and prove that logarithmic Hochschild homology is functorial with respect to the induced transforms. Unlike the classical setting, logarithmic correspondences do not naturally live on ordinary products, and the standard adjunction formalism fails because of blow-up discrepancies. We overcome these difficulties by constructing explicit unit- and counit-type morphisms that provide the necessary adjunction data without requiring an ambient dg category of logarithmic sheaves. As applications, we construct a dg bicategory of logarithmic correspondences in which logarithmic Hochschild homology and cohomology become categorical invariants. We also define logarithmic Chern characters and a logarithmic Euler pairing compatible with the logarithmic Fourier--Mukai formalism.
\end{abstract}

\section{Introduction}

In \cite{HABLICSEK2026127}, log Hochschild homology and cohomology were defined via Artin fans, extending the logarithmic Hochschild theory of Olsson \cite{olsson2024loghochschild}. The resulting theory provides a logarithmic analogue of Hochschild invariants that behaves naturally with respect to logarithmic structures and admits a logarithmic Hochschild--Kostant--Rosenberg decomposition.

Given a quasicompact, quasiseparated, weakly log separated, finite type log algebraic stack $X$, one considers its Artin fan $\af{X}$ and the corresponding logarithmic diagonal map
\[
\delta:X\to X\times_{\af{X}}X.
\]
The associated dg endofunctors of a dg enhancement $D(X)$ of the derived category of coherent sheaves on $X$
\[
\delta^*\delta_*,
\qquad
\delta^!\delta_*:D(X)\to D(X)
\]
define the logarithmic Hochschild homology, $\HHl{X}$, and cohomology theories, $\cHHl{X}$, of $X$. Here and throughout the paper, all functors are derived unless stated otherwise. The corresponding logarithmic Hochschild (co)homology groups are defined as the hypercohomology groups obtained by evaluating these endofunctors on the structure sheaf.

\begin{definition}
The $n$-th log Hochschild homology group of $X$ is
\[
\HH_{-n}^{\log}(X)
\coloneq 
R^n\Gamma(X,\HHl{X}(\OO_X))=R^n\Gamma(X, \delta^*\delta_*\OO_X).
\]
Similarly, the $n$-th log Hochschild cohomology group is
\[
\HH^n_{\log}(X)
\coloneq R^n\Gamma(X,\cHHl{X}(\OO_X))=R^n\Gamma(X, \delta^!\delta_*\OO_X).
\]
\end{definition}

In \cite[Theorem B]{HABLICSEK2026127}, the logarithmic Hochschild--Kostant--Rosenberg theorem is established using formality of derived self-intersections \cite{arinkin-caldararu}. If $X$ is a quasicompact, weakly log separated, log smooth log scheme over a field of characteristic $0$, loc.\ cit.\ proves that the dg endofunctor $\delta^*\delta_*$ is formal. More precisely, there is an isomorphism of dg endofunctors
\[
\delta^*\delta_*(-)
\simeq
(-)\otimes \Sym(\Omega_X^{1,\log}[1]).
\]
Here $\Sym$ denotes the derived symmetric algebra, which can be identified with the graded exterior algebra
\[
\Sym(\Omega_X^{1,\log}[1])
=
\bigoplus_q \wedge^q\Omega_X^{1,\log}[q].
\]
Consequently, there is a logarithmic HKR decomposition
\[
\HH_n^{\log}(X)
=
\bigoplus_{q-p=n}
H^p(X,\Omega_X^{q,\log}).
\]


In this paper, we focus on logarithmic schemes arising from smooth proper log pairs over a field of characteristic 0.

\begin{definition}\label{def:logpair}
A \emph{smooth proper log pair} $(X,D_X)$ consists of a smooth proper variety $X$ over a field of characteristic 0 together with a smooth divisor $D_X$ such that both $X$ and $D_X$ are geometrically connected.
\end{definition}

For such a pair, the logarithmic self-product
\[
X\times_{\af{X}}X
\]
admits a natural compactification inside the blow-up
\[
Bl_{D_X\times D_X}(X\times X).
\]
More precisely, $X\times_{\af{X}}X$ is the complement of the strict transforms of $D_X\times X$ and $X\times D_X$ inside the blow-up. The logarithmic diagonal map $\delta$ is the restriction of the ordinary diagonal embedding into the blow-up
\[
i:X\hookrightarrow Bl_{D_X\times D_X}(X\times X).
\]
We view the blow-up $Bl_{D_X\times D_X}(X\times X)$ as a proper logarithmic replacement of $X\times X$.

Since $X\times_{\af{X}}X$ is an open subvariety of the blow-up, the counit map
\[
i^*i_*\Rightarrow \delta^*\delta_*
\]
is an isomorphism. Hence, we may identify the endofunctors
\[
\delta^*\delta_*
=
i^*i_*,
\qquad
\delta^!\delta_*
=
i^!i_*.
\]

The conormal bundle of the embedding $i$ is the locally free sheaf of logarithmic one-forms $\Omega_X^1(\log D_X)$. Combining this observation with the logarithmic HKR theorem yields an isomorphism of dg endofunctors
\[
i^*i_*(-)
\simeq
(-)\otimes \Sym(\Omega_X^1(\log D_X)[1]),
\]
and, therefore, isomorphisms of vector spaces
\[\HH_n^{\log}(X)=\bigoplus_{q-p=n}H^p(X, \Omega^q_X(\log D_X)).\]

\medskip

A central open problem in logarithmic geometry is the construction of a suitable dg category of log coherent sheaves whose categorical invariants are logarithmic Hochschild homology and cohomology (\cite{mehta1980moduli, talpo2018infinite, scherotzke2020parabolic, vaintrob2017categorical}). 
Instead of directly attacking this problem, we initiate the study of functoriality in the logarithmic setting.

The main question of this paper is the following.

\medskip

\noindent
\textbf{Main Question.}
\emph{Does a Fourier--Mukai transform induce a natural map on logarithmic Hochschild homology?}

\medskip

Because of the absence of a suitable category of log coherent sheaves, instead of studying Fourier-Mukai transforms, we focus on the kernels of these transforms. For smooth log pairs $(X,D_X)$ and $(Y,D_Y)$, we consider Fourier--Mukai kernels on the logarithmic compactification
\[
B\coloneq Bl_{D_X\times D_Y}(X\times Y).
\]
This choice is forced by the geometry of the logarithmic setting: even the logarithmic diagonal, which should correspond to the identity functor, does not live on $X\times X$, but lives on the blow-up $Bl_{D_X\times D_X}(X\times X)$. Moreover, working on a proper logarithmic compactification is essential in order to obtain pushforward maps and the expected cohomological properties; see, for instance \cite{hablicsek2026logarithmic}.

Although our definition of logarithmic Fourier--Mukai kernels formally resembles the classical one, several genuinely new phenomena appear. The first difficulty arises from the lack of a satisfactory derived category of logarithmic coherent sheaves.
As a consequence, the standard categorical construction of left and right adjoints is unavailable. The second difficulty is that even after constructing the candidate adjoints by hand, the usual compatibilities between adjunction and composition fail in general because blow-up discrepancies obstruct the classical formalism.

To overcome these difficulties, we introduce a geometrically distinguished class of kernels, which we call \emph{strong log Fourier--Mukai kernels}. These are perfect complexes on
\[
Bl_{D_X\times D_Y}(X\times Y)
\]
supported away from the strict transforms of $D_X\times Y$ and $X\times D_Y$. Intuitively, strong kernels arise naturally from considering the Artin fans of the log smooth pairs that encode the additional boundary data required for functoriality in logarithmic geometry.

A key technical contribution of the paper is the construction of explicit unit- and counit-type morphisms that play the role of left and right adjunction data. Although these morphisms are not induced from an ambient category of logarithmic sheaves, they satisfy the compatibilities necessary to establish Hochschild functoriality. In particular, this approach circumvents the current absence of a satisfactory logarithmic derived category while still retaining the essential formal properties needed for the theory.

Our main theorem is the following.

\begin{mainthm}[{Theorem~\ref{thm:main}}]
Let
\[
\phi_E:D(X)\to D(Y)
\]
be a strong log Fourier--Mukai transform induced by a strong log Fourier--Mukai kernel
\[
E\in D(Bl_{D_X\times D_Y}(X\times Y)).
\]
Then $\phi_E$ induces a canonical contravariant map on logarithmic Hochschild homology
\[
\phi_E^{HH}:
\HH^{\log}_*(Y)\to \HH^{\log}_*(X),
\]
functorial with respect to composition of strong log Fourier--Mukai transforms.
\end{mainthm}

This theorem should be viewed as the logarithmic analogue of the classical statement that Fourier--Mukai transforms induce maps on Hochschild homology \cite{caldararu2003mukai, kuznetsov2009}. However, the logarithmic setting requires a substantially different geometric and categorical framework.

We apply our main theorem in several directions. First, we construct a dg bicategory of strong logarithmic Fourier--Mukai correspondences in which logarithmic Hochschild (co)homology becomes a categorical invariant, in the spirit of \cite{caldararu2010mukai}. 
\begin{corollary}[{Theorem~\ref{thm:dgbicat}}] There exists a dg-bicategory of smooth and proper log pairs so that log Hochschild homology and cohomology are invariants of this bicategory.
\end{corollary}
Second, we define logarithmic Chern characters and a logarithmic Euler pairing compatible with the Fourier--Mukai formalism.
\begin{corollary}[{Sections~\ref{sec:logChern} and~\ref{sec:logEuler}}] Let $(X,D)$ be a smooth log pair.
\begin{enumerate}
\item There exists a log Chern character of strong log Fourier-Mukai kernels on $X\times^{\log}X$.
\item There exists a log Chern character of strong log Fourier-Mukai kernels on $\mathbb{P}^1 \times^{\log} X$.
\item There exists a log Euler pairing of strong log Fourier-Mukai kernels on $\PP^1\times^{\log}X$.
\end{enumerate}
\end{corollary}

Conceptually, one of the main contributions of this paper is a shift in perspective on the search for the ``correct'' derived category in logarithmic geometry. In our framework, functoriality is governed not by ordinary morphisms of logarithmic schemes, but rather by logarithmic correspondences living on canonical compactifications. This suggests that logarithmic categorical geometry may be fundamentally correspondence-theoretic.

The rest of the paper is structured as follows. In Section~\ref{sec:blowup}, we study the geometry of the blow-up compactifications that underlie the logarithmic Fourier--Mukai formalism. In particular, we prove an analogue of the smooth base change theorem for sequential blow-ups. In Section~\ref{sec:logFM}, we introduce logarithmic Fourier--Mukai kernels and develop their basic properties. We define strong log Fourier--Mukai kernels and construct explicit unit- and counit-type morphisms playing the role of non-categorical left and right adjoints. Using this framework, we establish in Section~\ref{sec:functoriality} the functoriality of logarithmic Hochschild homology. Finally, in Section~\ref{sec:applications}, we discuss applications of the theory, including the construction of a dg bicategory of logarithmic correspondences, logarithmic Chern characters and logarithmic Euler pairings.

\textbf{Conventions:} There are two distinct notions of the log cotangent complex in the literature (due to Gabber and Olsson) providing two distinct notions of log Hochschild homology/cohomology (see also \cite{rognesoriginalloghh} and \cite{HKRlogHH}). However, the two notions of the log cotangent complex agree on simple cases, for instance, in the case of a log smooth pair. Therefore, we believe that our paper contributes to both communities. 

\textbf{Acknowledgments:}  M.H. was supported by the MTA Distinguished Guest Scientist Fellowship Programme 2026 and part of the work was done at the Budapest University of Technology and Economics. \'A.Gy. was supported by the János Bolyai Research Scholarship of the Hungarian Academy of Sciences. We thank Patrick Kennedy-Hunt for inspiring discussions. For instance, the notion of a strong Fourier-Mukai kernel was an outcome of these discussions. 

\section{World of blow-ups}\label{sec:blowup}

In this section, we follow \cite{li-li} to define sequential blow-ups of varieties. We analyze the relationship between them. Throughout the section, we consider smooth, proper log pairs as in Definition \ref{def:logpair} in characteristic zero. The Artin fan of any log pair is $\af{} = \bra{\Aff^1/\GG_m}$. 

In this paper, the sequential blow-ups play the role of logarithmic replacements, $X_1\times^{\log}...\times^{\log}X_n$, of products $X_1\times...\times X_n$ of log pairs $(X_i, D_{X_i})_{i=1}^n$. These replacements are smooth and proper varieties. We emphasize that we could have chosen to work with the open subvarieties $X_1\times_{[\Aff^1/\GG_m]}...\times_{[\Aff^1/\GG_m]}X_n$ given by the Artin fan $[\Aff^1/\GG_m]$ of smooth log pairs. Most of the statements on the logarithmic replacements are motivated by the formalism with Artin fans. Nevertheless, we prefer to work on smooth and proper varieties.

\begin{remark}
    Because log Hochschild (co)homology is invariant under log alteration (see \cite{HABLICSEK2026127} and \cite{hablicsek2026logarithmic}), it is natural to expect that any logarithmic replacement should work in our construction. However, given that we would like to track functoriality explicitly, we use a fixed logarithmic replacement throughout the paper.
\end{remark}

\begin{definition} 
We say that a finite subset $\mathcal{S}$ of smooth subvarieties of a smooth variety $X$ is an \emph{arrangement} if 
\begin{itemize}
\item any pair, $S_i, S_j\in \sS$ of smooth subvarieties of $\mathcal{S}$ intersect cleanly (meaning that $S_i\cap S_j$ is again a smooth subvariety so that for any point $P\in S_i\cap S_j$, the tangent space $T_p(S_i\cap S_j)=T_p(S_i)\cap T_p(S_j)$), and
\item the intersection $S_i\cap S_j$ is also in $\sS$.
\end{itemize}
\end{definition}

\begin{definition}
Let $\sS$ be an arrangement on a smooth variety $X$. We say that a subset $\mathcal{G}\subset \sS$ is a \textit{building set} of $\sS$ if $\forall S\in \mathcal{S}$, the minimal elements in $\{G\in \mathcal{G} : S\subseteq G\}$ intersect transversally and the intersection is $S$. Furthermore, a set of subvarieties $\mathcal{G}$ is called a building set if all the possible intersections of subvarieties in $\mathcal{G}$ form an arrangement $\sS$ and $\mathcal{G}$ is a building set of $\sS$.
\end{definition}

As an example, consider smooth log pairs $(X_1, D_{X_1})$, ..., $(X_n, D_{X_n})$ and take $X=X_1\times...\times X_n$. Consider all codimension two smooth subvarieties of the form $D_{i,j} \coloneqq X_1\times X_2\times... \times D_{X_i}\times...\times D_{X_j}\times...\times X_n$ inside $X_1\times X_2\times...\times X_n$, and all subvarieties that are intersections of the $D_{i,j}$. These smooth subvarieties form an arrangement, 
so $\mathcal{D}_n$ itself is a building set.

\begin{definition}
Let $(X_1, D_{X_1})$, ..., $(X_n, D_{X_n})$ be log smooth pairs. We define the \emph{log product} $X_1\times^{\log} X_2\times^{\log}...\times^{\log} X_n$ as the sequential blow-up of $X_1\times X_2\times... \times X_n$ via the building set $\mathcal{D}_n$ given in the example above, starting at the highest codimensional strata. This corresponds to the barycentric subdivision of the Artin fan $\af{}^n = \bra{\Aff^n/\GG_m^n}$ of the product $X_1 \times \cdots \times X_n$. 
\end{definition}

This sequential blow-up is isomorphic to the blow-up, $Bl_{I_1...I_k}(X_1\times...\times X_n)$, where $I_i$ denotes the ideal sheaf of the $i$-th element of the building set $\mathcal{D}_n$. Furthermore, the sequential blow-ups $X_1\times^{\log}X_2\times^{\log}...\times^{\log}X_n$ are smooth varieties (\cite{li-li}).

\begin{example}
If $(X,D_X)$ and $(Y, D_Y)$ are smooth log pairs, then $X\times^{\log} Y$ is defined as the blow-up of $X\times Y$ along $D_X\times D_Y$. 
\end{example}

\begin{example}\label{ex:blowup3}
If $(X, D_X)$,  $(Y, D_Y)$ and $(Z,D_Z)$ are smooth log pairs, then $X\times^{\log}Y\times^{\log}Z$ is given by a sequential blow-up, when first we blow-up $D_X\times D_Y\times D_Z$, then the strict transforms of $D_X\times D_Y\times Z$, $D_X\times Y\times D_Z$ and finally of $X\times D_Y\times D_Z$. The order in which we take the blow-up of the 3 strict transforms does not matter. 

We can describe the same blow-up differently. Given the ordering $S_1, S_2, S_3, S_4$ of $\mathcal{D}_3$ given as $S_1=D_X\times D_Y\times Z, S_2=D_X\times D_Y\times D_Z, S_3=D_X\times Y\times D_Z, S_4=X\times D_Y\times D_Z$, we see that for any $1\leq i\leq 4$, $S_1, ..., S_i$ form a building set, hence $X\times^{\log}Y\times^{\log}Z$ can be identified as the sequential blow-up with this order \cite{li-li}. We emphasize that while, in this case, we blow-up the subvarieties in different order, the resulting blow-up is the same as before.

Consider the natural map $X\times^{\log}Y\times^{\log}Z\to X\times^{\log}Y\times Z$ (here the latter space is obtained by only blowing up $D_X\times D_Y\times Z$). The three subvarieties $D_X\times D_Y\times D_Z, X\times D_Y\times D_Z$, and $D_X\times Y\times D_Z$ form a building set. The log product $X\times^{\log}Y\times^{\log}Z$, is obtained by successively blowing up their dominant transforms. 
\end{example}

\begin{example}
    If $X, Y, Z$ are $\Aff^1$ with boundary $D = \vec 0$, the log product $\Aff^1 \times^{\log} \Aff^1 \times^{\log} \Aff^1$ is the blowup of $\Aff^3$ at the origin $\vec 0 \in \Aff^3$ and then at the strict transforms of the three axes. The result is the toric variety depicted in Figure \ref{fig:barycentricA3} corresponding to the barycentric subdivision of the octant, the normal fan of the permutohedron in three dimensions. 
\end{example}

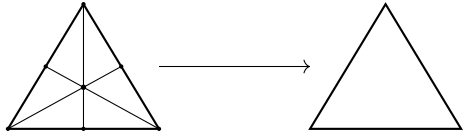
\begin{figure}
    \centering
\begin{tikzpicture}[scale=2]

\coordinate (A) at (0,0);
\coordinate (B) at (1,0);
\coordinate (C) at (0.5,0.82);

\coordinate (AB) at ($(A)!0.5!(B)$);
\coordinate (BC) at ($(B)!0.5!(C)$);
\coordinate (CA) at ($(C)!0.5!(A)$);

\coordinate (G) at (.5,{(0+0+0.82)/3});

\draw[thick] (A)--(B)--(C)--cycle;


\draw (G)--(A);
\draw (G)--(B);
\draw (G)--(C);
\draw (G)--(AB);
\draw (G)--(BC);
\draw (G)--(CA);




\fill (A) circle (0.4pt);
\fill (B) circle (0.4pt);
\fill (C) circle (0.4pt);

\fill (AB) circle (0.4pt);
\fill (BC) circle (0.4pt);
\fill (CA) circle (0.4pt);

\fill (G) circle (0.5pt);

\draw[->] (1, .41) to (2, .41);

\begin{scope}[shift = {(2, 0)}]
\draw[thick] (0, 0)--(1, 0)--(.5, 0.82)--cycle;
\end{scope}

\end{tikzpicture}
    \caption{The triple log product $\Aff^1 \times^{\log} \Aff^1 \times^{\log} \Aff^1$ is the toric variety with fan the barycentric subdivision of the triangle, viewed as the height-one slice of the fan of $\Aff^3$. }
    \label{fig:barycentricA3}
\end{figure}

\begin{example}
    We describe \'etale locally how the charts of the sequential blow-up $X\times^{\log}Y\times^{\log}Z$ look like. Since we work \'etale locally, we assume that $X$, $Y$ and $Z$ are affine spaces, and the divisors $D_X$, $D_Y$ and $D_Z$ are cut out by regular functions $x$, $y$ and $z$. We first blow-up $D_X\times D_Y\times D_Z$ or in other words, we blow-up the ideal $(x,y,z)$ using the Rees construction (in geometric terms we obtain a projective plane as the exceptional divisor). The blow-up is covered by three standard affine charts, which we call the $x$-, $y$-, and $z$-charts corresponding to the chosen homogeneous coordinate. We describe now one of the standard affine charts, the $z-$chart. On this chart, we have coordinates $a_1=x/z$ and $a_2=y/z$. Over this chart, the only relevant blow-up we need to consider is the blow-up of the ideal $(a_1,a_2)$ which is the strict transform of $D_X\times D_Y\times Z$. The blow-up gives us a projective line. It has the two standard affine charts; the $a_1$- and the $a_2$-chart. Take the $a_2$-chart. Then, we have coordinates $t_3=z$, $t_2=a_2$, and $t_1=a_1/a_2$. With these blow-up coordinates, the original coordinates are given by
    \[x=t_1t_2t_3,\quad y=t_2t_3,\quad z=t_3.\]
    Note that this is only one of the six affine charts of $X\times^{\log}Y\times^{\log}Z$.
\end{example}

\begin{remark}\label{rem:locallogprod}
    This description holds in general. Let $(X_1, D_1)$, ..., $(X_n, D_n)$ be log smooth pairs. Consider the log product $X_1\times^{\log}...\times^{\log}X_n$. We work \'etale locally and assume that each $D_i$ is cut out by an equation $x_i=0$. Given a permutation $\sigma\in S_n$, we have a nested chain of strata
\[(x_{\sigma(1)}=...=x_{\sigma(n)}=0)\subset (x_{\sigma(1)}=...=x_{\sigma(n-1)}=0)\subset...\subset (x_{\sigma(1)}=0)\]
that tells us which affine charts we take in the sequential blow-up. Blowing up the smallest chart, and taking the coordinate $t_n=x_{\sigma(n)}$, we write $x_{\sigma(i)}=a_it_n$ (as above). Blowing up the ideal $(a_1,...,a_{n-1})$ and taking the coordinate $t_{n-1}=a_{n-1}$, we write $a_i=b_it_{n-1}$. We can continue the process till we blow up the smallest stratum. With this process we obtain local coordinates on the blow-up, $t_1$, ..., $t_n$ so that
\[x_{\sigma(i)}=t_i\cdot ...\cdot t_n.\]
We note that we ignored the smooth coordinates of the $X_i$, those contribute other affine coordinates of the blow-up.
\end{remark}

We begin to analyze the world of sequential blow-ups.



\begin{lemma}\label{lem:projection}
Let $(X_1, D_{X_1})$, ..., $(X_n, D_{X_n})$ be smooth log pairs. For any integer $m\leq n$, we have a natural map $X_1\times^{\log}...\times^{\log}X_n\to X_1\times^{\log}...\times^{\log}X_m$ extending the projection map $X_1\times...\times X_n\to X_1\times...\times X_m$, i.e, the following diagram is commutative.
\[
\begin{tikzcd}
X_1\times^{\log}...\times^{\log}X_n\ar[r]\ar[d] & X_1\times X_2\times...\times X_n\ar[d]\\
X_1\times^{\log}...\times^{\log}X_m\ar[r] & X_1\times X_2\times...\times X_m.
\end{tikzcd}
\]
\end{lemma}

\begin{proof}
Consider the building set $\mathcal{D}_n$ of $X_1\times...\times X_n$ and the building set $\mathcal{D}_m$ of $X_1\times...\times X_m$. We see that $\mathcal{D}_n$ contains all smooth subvarieties of the form $S\times X_{m+1}\times...\times X_n$ where $S\in \mathcal{D}_m$. Therefore, we obtain a natural map 
\[X_1\times^{\log}X_2\times^{\log}...\times^{\log}X_n\to X_1\times^{\log}X_2\times^{\log}...\times^{\log}X_m\times X_{m+1}\times...\times X_n.\]
Combining the map above with the projection map 
\[X_1\times^{\log}X_2\times^{\log}...\times^{\log}X_m\times X_{m+1}\times...\times X_n\to X_1\times^{\log}X_2\times^{\log}...\times^{\log}X_m\]
we obtain our natural map $X_1\times^{\log}...\times^{\log}X_n\to X_1\times^{\log}...\times^{\log}X_m$ that clearly extends the projection map $X_1\times...\times X_n\to X_1\times...\times X_m$.
\end{proof}

Similarly to the case of the usual products and projection maps, we have the following lemma.

\begin{lemma}\label{lem:flatprojection}
    The map $X_1\times^{\log}...\times^{\log}X_n\to X_1\times^{\log}...\times^{\log}X_m$ extending the projection map is flat.
\end{lemma}

\begin{proof}
    Note that it is enough to show this lemma in the case when $m=n-1$, and thus, we assume that $m=n-1$. Furthermore, it is enough to show the statement \'etale-locally, and hence, we assume that each $X_i$ is an affine space and the divisor is cut out by a single equation $x_i$. The affine charts of the sequential blow-up $X_1\times^{\log}...\times^{\log}X_n$ can be computed via the standard affine charts corresponding to the Rees algebra. 
    
    As discussed in Remark \ref{rem:locallogprod}, the affine charts are parametrized by permutations $\sigma\in S_n$.
    On the chart corresponding to $\sigma$, we have coordinates $t_1,...,t_n$ so that
    \[x_{\sigma(i)}=t_i...t_{n}.\]
    We again ignore the smooth coordinates, they only contribute with a factor of a polynomial ring (hence they can be ignored).
    
    Let us denote $\sigma^{-1}(n)$ by $p$. Deleting $n$ from $\sigma$ gives a permutation $\tilde{\sigma}\in S_{n-1}$ and therefore an affine chart of the sequential blow-up
    \[X_1\times^{\log}...\times^{\log}X_{n-1}.\]
    We denote the coordinates of this affine chart by $u_1,...u_{n-1}$. Then, the map on coordinate rings of the map $X_1\times^{\log}...\times^{\log}X_n\to X_1\times^{\log}...\times^{\log}X_{n-1}$ is given by either
    \[u_j\mapsto t_{j+1}\]
    (for all $1\leq j\leq n-1$) in the case when $p=1$ or
    \[u_j\mapsto \begin{cases} t_j, & j<p-1,\\ t_{p-1}t_p, & j=p-1,\\ t_{j+1}, & j\geq p. \end{cases}\]
    in the case when $p\ne 1$.
    
    Thus, the only non-trivial part of the map is given by the ring homomorphism $k[u]\longrightarrow k[s,t]$ that sends $u$ to $st$. That map is flat, finishing our proof.
\end{proof}

A crucial observation regarding our paper is the following.

\begin{remark}
    The proof of Lemma \ref{lem:flatprojection} shows that, in general, the maps 
    \[X_1\times^{\log}...\times^{\log}X_n\to X_1\times^{\log}...\times^{\log}X_m\] defined above are not smooth. Indeed, the ring map
    \[k[x]\to k[y,u]\]
    sending $x$ to $uy$ is flat, but not smooth.
\end{remark}

Given smooth log pairs $(X_1, D_1)$, ..., $(X_n, D_n)$, and $(Y, D_Y)$ and $(Z, D_Z)$, the composite maps (given by Lemma \ref{lem:projection}) 
\[X_1\times^{\log}...\times^{\log}X_n\times^{\log} Y\times^{\log}Z\to X_1\times^{\log}...\times^{\log}X_n\times^{\log}Y\to X_1\times^{\log}...\times^{\log}X_n\]
and
\[X_1\times^{\log}...\times^{\log}X_n\times^{\log}Y\times^{\log}Z\to X_1\times^{\log}...\times^{\log}X_n\times^{\log}Z\to X_1\times^{\log}...\times^{\log}X_n\]
are equal providing a commutative diagram
\[
\begin{tikzcd}
    X_1\times^{\log}...\times^{\log}X_n\times^{\log}Y\times^{\log}Z\ar[r, "q"]\ar[d, "p"] & X_1\times^{\log}...\times^{\log}X_n\times^{\log} Y\ar[d, "i"]\\ X_1\times^{\log}...\times^{\log}X_n\times^{\log}Z\ar[r, "j"]& X_1\times^{\log}...\times^{\log}X_n.
\end{tikzcd}
\]

Even though this commutative diagram is not a Cartesian diagram, we have the following important statement.

\begin{lemma}\label{lem:gensquare}
    The functors
    \[j^*i_*:D(X_1\times^{\log}...\times^{\log}X_n\times^{\log} Y)\to D(X_1\times^{\log}...\times^{\log}X_n\times^{\log} Z)\]
    and 
    \[ p_*q^*:D(X_1\times^{\log}...\times^{\log}X_n\times^{\log} Y)\to D(X_1\times^{\log}...\times^{\log}X_n\times^{\log} Z)\]
    are equivalent.
\end{lemma}

\begin{proof}
The proof has three essential steps:
\begin{itemize}
    \item we show that the ordinary pullback, $B$, provides a derived Cartesian square,
    \item we show that the ordinary pullback has mild singularities,
    \item we show that the natural map $X_1\times^{\log}...\times^{\log}X_n\times^{\log}Y\times^{\log}Z$ is a resolution of singularity map.
\end{itemize}
We start with the first step. We consider the ordinary pullback 
\[
    B\coloneq  (X_1\times^{\log}...\times^{\log}X_n\times^{\log} Y) \times _{X_1\times^{\log}...\times^{\log}X_n} (X_1\times^{\log}...\times^{\log}X_n\times^{\log}Z)\]
and the derived pullback
\[
    B^R\coloneq  (X_1\times^{\log}...\times^{\log}X_n\times^{\log} Y) \times^R _{X_1\times^{\log}...\times^{\log}X_n} (X_1\times^{\log}...\times^{\log}X_n\times^{\log}Z).
\]
The derived fiber product is a derived scheme that is given by gluing local affine charts together where the tensor product is replaced by the derived tensor product (see \cite{toen2014derived} for more details). The derived pullback admits a natural map, the truncation, from the underived pullback $B\to B^R$ that comes from the universal property of the derived Cartesian square.

Since the map $i$ (and similarly the map $j$) is flat from Lemma \ref{lem:flatprojection}, we obtain that the truncation map is an isomorphism, meaning that the derived and the ordinary (underived) pullbacks agree. This finishes the first step of the proof.

For the second step, we compute local affine charts of $B$. We assume that we work \'etale locally, and the varieties $X_1$, ..., $X_n$, $Y$ and $Z$ are affine spaces with smooth divisors given by regular functions $x_1$, ..., $x_n$, $y$ and $z$. We will not write out the coordinates of the spaces for sake of simplicity and that they only contribute  smooth factors. Fix a permutation $\sigma\in S_n$. Using the discussion in Remark \ref{rem:locallogprod}, the corresponding affine chart of $X_1\times^{\log}...\times^{\log}X_n$ has coordinates given by $t_1,...,t_n$ satisfying
\[x_{\sigma(i)}=t_i...t_{n}.\]
The charts of $X_1\times^{\log}...\times^{\log}X_n\times^{\log}Y$ that lie over this chart are obtained by inserting a $y$ into one of the $n+1$ positions of the permutation $\sigma$. Assume we put $y$ into position $p$. Let us denote the coordinates of this chart by $s_1,...,s_{n+1}$. Then, the projection map 
\[X_1\times^{\log}...\times^{\log}X_n\times^{\log}Y\to X_1\times^{\log}...\times^{\log}X_n\]
is given by the map $t_j\mapsto s_{j+1}$ if $p=1$ or
\[t_i\mapsto \begin{cases} s_i, & i<p-1,\\ s_{p-1}s_p, & i=p-1,\\ s_{i+1}, & i\geq p\end{cases}\]
if $p>1$. Similarly, the projection 
\[X_1\times^{\log}...\times^{\log}X_n\times^{\log}Z\to X_1\times^{\log}...\times^{\log}X_n\]
can be described by inserting $z$ into position $p'$. The coordinates of the chart (which we denote by $u_1,...,u_{n+1}$) and the projection map algebraically can be described analogously as above.

If $p=1$, then $t_1\mapsto s_2$, and hence, we obtain a free algebra extension, and thus, the fiber product is smooth. Similar argument holds for $p'=1$. 

If $p\ne p'$ (and $p,p'>1$), then the projections replace two different coordinates $t_{p-1}$ and $t_{p'-1}$ respectively with products of coordinates. Taking $p<p'$, this means that the fiber product has coordinates $s_1=u_1,...,s_{p-1}s_p=u_{p-1},...,s_{p'-1}=u_{p'-1}u_{p'},...,s_{n+1}=u_{n+1}$ that can be simplified to coordinates $s_1,...,s_p,u_{p+1},..., u_{p'-1},...,u_{n+1}$, in other words, the fiber product is smooth.

In the case of $p=p'$ (and $p,p'>1$), the situation is different. The assignments $t_{p-1}\mapsto s_{p-1}s_p$ and $t_{p-1}\mapsto u_{p-1}u_p$ cannot be simplified as above, in this case the coordinate ring of the fiber product has a factor of the form
\[k[s_{p-1}, s_p, u_{p-1},u_p]/(s_{p-1}s_p-u_{p-1}u_p)\]
which gives a singular variety. The other factors are smooth as above. Therefore, \'etale locally, $B$ is either smooth or the product of a smooth
variety with
\[k[s_{p-1}, s_p, u_{p-1},u_p]/(s_{p-1}s_p-u_{p-1}u_p).\]
This implies that $B$ has rational singularities (\cite{fulton1993introduction}) finishing step 2 of the lemma. 

Now, we turn our attention to the third step of the lemma.

We have a natural (global) map  \[\pi:X_1\times^{\log}...\times^{\log}X_n\times^{\log}Y\times^{\log}Z\to B\]
coming from the universal property of Cartesian product.
\[\Tiny
\begin{tikzcd}
X_1\times^{\log}...\times^{\log}X_n\times^{\log}Y\times^{\log}Z\ar[rrd, "q"]\ar[rdd, "p"] \ar[rd, "\pi"] & & \\
      & B\ar[r, "q' "]\ar[d, "p' "] & X_1\times^{\log}...\times^{\log}X_n\times^{\log}Y\ar[d, "i"]\\ & X_1\times^{\log}...\times^{\log}X_n\times^{\log}Z\ar[r, "j"]& X_1\times^{\log}...\times^{\log}X_n.
\end{tikzcd}
\]
Note that the maps $i$ and $j$ are proper because they are composite of sequential blow-ups and projection maps. Consider the map $q=q'\circ \pi$. The map $q'$ is also proper (since it is the base change of the proper map $j$), and so is $q$ (it is a sequential blow-up composed with a projection). As $q'$ is separated, $\pi$ is proper. 

Furthermore, $\pi$ is a birational map, it is the identity on the open complement of the divisor on $X_1\times^{\log}...\times^{\log}X_n\times^{\log}Y\times^{\log}Z$. Since $X_1\times^{\log}...\times^{\log}X_n\times^{\log}Y\times^{\log}Z$ is smooth \cite{li-li}, the map $\pi$ is a resolution of singularity map finishing the third step of the proof. 

Since $\pi$ is a resolution of singularities map and $B$ only has rational singularities, we get that 
\[\pi_*\OO_{X_1\times^{\log}...\times^{\log}X_n\times^{\log}Y\times^{\log}Z}=\OO_B.\]
We emphasize again that our functors are derived, so this equality also means that the higher direct images are trivial. 
This means that the functor $\pi_*\pi^*$ is equivalent to the identity functor on perfect complexes, indeed
\[\pi_*\pi^*F\simeq F\otimes \pi_*\OO_{X_1\times^{\log}...\times^{\log}X_n\times^{\log}Y\times^{\log}Z}=F\]
by the projection formula.

We have a sequence of equivalences of functors
\[j^*i_*\simeq p'_*q'^{*}\simeq p'_*\pi_*\pi^*q'^{*}\simeq p_*q^*\]
where the first equivalence follows from the fact that base change theorem holds for derived pullback diagrams \cite{toen2012proper} and the second equivalence follows from the fact that $\pi_*\pi^*$ is equivalent to the identity functor.
\end{proof}

\begin{remark}
    In the case of $n=1$, the local computation is very explicit. We assume that we work \'etale locally and $X$ is an affine space with coordinate ring $k[\underline{x}]$ with divisor given by $x=0$; $Y$ is an affine space with coordinate ring $k[\underline{y}]$ with divisor given by $y=0$, and finally, $Z$ is an affine space with coordinate ring $k[\underline{z}]$ with divisor given by $z=0$. Then, the local charts of $X\times^{\log}Y$ are given by the coordinate rings
    \[k[\underline{x},\underline{y}, u]/(x-uy)\quad\mbox{or}\quad k[\underline{x}, \underline{y}, v]/(y-vx).\]
    Similarly, the charts of $X\times^{\log}Z$ are given by the coordinate rings
    \[k[\underline{x},\underline{z}, u']/(x-u'z)\quad\mbox{or}\quad k[\underline{x}, \underline{z}, v']/(z-v'x).\]
    Thus, $B$ can be described with four charts. Three of these four charts are smooth:
    \[k[\underline{x}, \underline{y}, \underline{z}, u, v']/(x-uy, z-v'x), \quad k[\underline{x}, \underline{y}, \underline{z}, u', v]/(y-vx, x-u'z),\]
    \[k[\underline{x}, \underline{y}, \underline{z}, v, v']/(y-vx, z-v'x).\]
    Finally the last chart of the four charts is singular:
    \[k[\underline{x}, \underline{y}, \underline{z}, u, u']/(x-uy, x-u'z).\]
\end{remark}

Lemma \ref{lem:gensquare} has an important consequence which can be seen as a base change theorem for log schemes.  Let $I$ and $J$ be finite index sets so that $I\cap J\ne \emptyset$. Let $\{(X_i, D_i)\}_{i\in I}$ be a finite set of log smooth pairs parametrized by the index set $I$ and $\{(X_j, D_j)\}_{j\in J}$ be a finite set of log smooth pairs parametrized by the index set $J$ (if $i\in I\cap J$, then the corresponding log smooth pairs are identical). Then, we have a commutative diagram given by Lemma \ref{lem:projection}
 \[
    \begin{tikzcd}
    \times^{\log}_{l\in I\cup J} X_l \ar[r, "q"]\ar[d, "p"] & \times^{\log}_{i\in I} X_i\ar[d, "i"]\\ \times^{\log}_{j\in J}X_j\ar[r, "j"]& \times^{\log}_{k\in I\cap J}X_k.
    \end{tikzcd}
    \]

\begin{theorem}\label{thm:generalsquare}
    With the notation above, we have that the functors
    \[j^*i_*:D(\times^{\log}_{i\in I} X_i)\to D(\times^{\log}_{j\in J} X_j)\]
    and 
    \[ p_*q^*:D(\times^{\log}_{i\in I} X_i)\to D(\times^{\log}_{j\in J} X_j)\]
    are equivalent.
\end{theorem}

\begin{proof}
    This is a straightforward consequence of Lemma \ref{lem:gensquare}. We prove the statement in the case of a two step tower, the general proof is obtained via induction.
    \[
    \begin{tikzcd}    X_1\times^{\log}...\times^{\log}X_n\times^{\log}Y\times^{\log}Z\times^{\log}W \ar[r, "r"]\ar[d, "s"] & X_1\times^{\log}...\times^{\log}X_n\times^{\log}Y\times^{\log}W\ar[d, "g"]\\ 
    X_1\times^{\log}...\times^{\log}X_n\times^{\log}Y\times^{\log}Z \ar[r, "q"]\ar[d, "p"] & X_1\times^{\log}...\times^{\log}X_n\times^{\log}Y\ar[d, "i"]\\ X_1\times^{\log}...\times^{\log}X_n\times^{\log}Z\ar[r, "j"]& X_1\times^{\log}...\times^{\log}X_n.
    \end{tikzcd}\]
    We wish to show that the functors $j^*(i\circ g)_*$ and $(p\circ s)_*r^*$ are equivalent. From Lemma \ref{lem:gensquare}, we know that the pairs of functors
    \[q^*g_*\quad\mbox{and}\quad s_*r^*\]
    and
    \[j^*i_*\quad\mbox{and}\quad p_*q^*\]
    are equivalent. This implies that
    \[j^*(i\circ g)_*\simeq j^*i_*g_*\simeq p_*q^*g_*\simeq p_*s_*r^*\simeq (p\circ s)_*r^*.\]
    This proves our claim.
\end{proof}

\begin{remark}
    The statement of Theorem \ref{thm:generalsquare} is motivated via Artin fans. In fact, we have open subsets of the 4 spaces involved given by products over Artin fan, $[\Aff^1/\GG_m]$, that form a commutative, Cartesian diagram.
 \[
    \begin{tikzcd}
    (\times_{[\Aff^1/\GG_m]}X_l)_{l\in I\cup J} \ar[r, "q"]\ar[d, "p"] & (\times_{[\Aff^1/\GG_m]}X_i)_{i\in I} \ar[d, "i"]\\ (\times_{[\Aff^1/\GG_m]}X_j)_{j\in J}\ar[r, "j"]& (\times_{[\Aff^1/\GG_m]}X_k)_{k\in I\cap J}.
    \end{tikzcd}
    \]
    The analogous statement on the isomorphisms of dg functors (Theorem \ref{thm:generalsquare}) follows from the fact that this Cartesian diagram is also derived Cartesian. We emphasize that while the open subsets form a derived Cartesian square, the compactified theorem is much more involved.
\end{remark}

\begin{remark}
    We note that Theorem \ref{thm:generalsquare} is in essence very similar to Theorem F of \cite{dell2026coherent} and suggests that the machinery there may be adapted to our setting.
\end{remark}

\begin{remark}
    In general, we cannot expect the base-change theorem for an arbitrary logarithmic alteration to hold. In \cite{bozec2025koszul}, it is shown that the commutative blow-up diagram
    \[
    \begin{tikzcd}
        E\ar[r, "q"]\ar[d, "p"]& B\ar[d, "i"]\\
        Z\ar[r, "j"] & X
    \end{tikzcd}
    \]
    (where $X$ is a smooth scheme, $Z$ is a smooth codimension at least 2 subvariety, $B$ is the blow-up of $Z$ inside $X$ and $E$ is the exceptional divisor) does not satisfy base change, the functors $i^*j_*$ and $q_*p^*$ are not equivalent.
\end{remark}

Now, we show a similar theorem as Theorem \ref{thm:generalsquare}.

Consider the following commutative diagram of ordinary schemes (without the log structure)
\begin{equation}\label{dia:cart}
\begin{tikzcd}
    X\times^{\log} Y\ar[r, "j"]\ar[d, "q"] & X\times^{\log}Y\times^{\log} X\ar[d, "p"]\\
    X\ar[r, "i"] & X\times^{\log} X.
\end{tikzcd} \end{equation}
Here the maps $p$ and $q$ are the natural projections given by Lemma \ref{lem:projection}, $i$ is the log diagonal map and $j$ is the log diagonal map on the $X$ component (and identity on the $Y$ component).

\begin{lemma}\label{lem:square2}
    The functors
    \[i^*p_*:D(X\times^{\log}Y\times^{\log}X)\to D(X)\]
    and
    \[q_*j^*:D(X\times^{\log}Y\times^{\log}X)\to D(X)\]
    are equivalent.

    Furthermore, the functors
    \[p^*i_*:D(X)\to D(X\times^{\log}Y\times^{\log}X)\]
    and
    \[j_*q^*: D(X)\to D(X\times^{\log}Y\times^{\log}X)\]
    are also equivalent.

    In fact, Diagram \ref{dia:cart} is Cartesian and derived Cartesian.
\end{lemma}

\begin{proof}
   Let us denote the ordinary fiber product 
   \[X\times_{X\times^{\log}X}(X\times^{\log}Y\times^{\log}X)\]
   by $B$. The universal property of Diagram \ref{dia:cart} induces a canonical map 
   \[\pi:X\times^{\log}Y\longrightarrow B.\]
   We will show that the map $\pi$ is an isomorphism. 
   
   To do so, we work \'etale locally. We assume that $X$ and $Y$ are affine spaces and the divisor is cut out by a single equation. We use the notation of $x_1$ and $x_2$ to differentiate between the two copies of $X$ appearing in the Diagram \ref{dia:cart}, and we write $y$ for the equation that cuts out $D_Y$. We ignore the smooth coordinates again (they contribute an extra factor of a polynomial ring).
   
   We consider the affine chart of $X\times^{\log}X$ with coordinates
    $a,b$ given by
    \[x_1=ab,\qquad x_2=b.\]
    On this chart, the logarithmic diagonal $i:X\to X\times^{\log}X$ is
    given by
    \[a\mapsto 1,\qquad b\mapsto x.\]
    Indeed, in terms of the coordinate rings, the variables corresponding to $x_1$ and $x_2$ are mapped to $x$. 

    There are three affine charts of $X\times^{\log}Y\times^{\log}X$ over this chart, corresponding to the three possible positions of $Y$ with respect to the two copies of $X$.
    \begin{itemize}
        \item If $Y$ is in the first position, then we have the coordinates $a, b, c$ so that
        \[y=abc, \quad x_1=ab, \quad x_2=b.\]
        In this case, the fiber product, $B$, is given by the coordinate ring (modulo the smooth coordinates) that is the polynomial ring $k[b,c]$. So in this case, $B$ is smooth.
        \item If $Y$ is in the second position, we have coordinates $r, s, b$ so that
        \[x_1=rsb,\quad y=sb,\quad x_2=b.\]
        In this case, the fiber product, $B$, is given by the coordinate ring
        \[k[r,s,b]/(rs-1)\]
        that is also smooth.
        \item If $Y$ is in the third position, we have coordinates $r,s,a$ so that
        \[x_1=ars,\quad x_2=rs,\quad y=s.\]
        In this case, the fiber product, $B$, is given by the coordinate ring $k[r,s]$ which is also smooth.
    \end{itemize}
    To summarize the above, we get that $B$ is smooth. Furthermore, in the above discussion the first and third charts are the two standard affine charts of the blow-up $Bl_{D_X\times D_Y}(X\times Y)$, i.e of $X\times^{\log}Y$. The second chart is the intersection of these two charts. This shows that locally the map $\pi:X\times^{\log}Y\to B$ is an isomorphism, hence it is also globally an isomorphism. 
    
    This implies that Diagram \ref{dia:cart} is Cartesian (unlike in the case of Lemma \ref{lem:gensquare}). Furthermore, the projection $p$ is flat by Lemma \ref{lem:flatprojection}, therefore the diagram is also derived Cartesian. This implies the equivalences of functors of the statement.
\end{proof}

\begin{remark}
    The lemma above has a natural interpretation via Artin fans. Consider the commutative diagram consisting of open subsets of the blow-ups
    \[
    \begin{tikzcd}
        X\times_{[\Aff^1/\GG_m]}Y\ar[r]\ar[d] & X\times_{[\Aff^1/\GG_m]}Y\times_{[\Aff^1/\GG_m]}X\ar[d]\\
        X\ar[r] & X\times_{[\Aff^1/\GG_m]}X.
    \end{tikzcd}
    \]
    It is straightforward to see that this commutative diagram is Cartesian. The statement on the isomorphisms on the dg functors follows from the fact that the diagram is also derived Cartesian.
\end{remark}

To conclude this section, we generalize Lemma \ref{lem:square2}. We consider log smooth pairs $(X, D), (Y_1, D_1), ..., (Y_n, D_n), (Z, D_Z)$. We consider the commutative diagram
\[
\begin{tikzcd}
    X\times^{\log}Y_1\times^{\log}...\times^{\log}Y_n\times^{\log}Z\ar[r, "\tilde{j}"]\ar[d, "\tilde{q}"] & X\times^{\log} X\times^{\log}Y_1\times^{\log}...\times^{\log}Y_n\times^{\log}Z\ar[d, "\tilde{p}"]\\
    X\times^{\log}Y_1\times^{\log}...\times^{\log}Y_n \ar[r, "\tilde{i}"] & X\times^{\log} X\times^{\log}Y_1\times^{\log}...\times^{\log}Y_n
\end{tikzcd}\]
where the horizontal arrows are given by the log diagonal maps on the copies of $X$ and identity otherwise, and the vertical maps are given by projection maps.

\begin{lemma}\label{lem:square2general}
    With this notation, we have that the functors
    \[\tilde{i}^*\tilde{p}_*:D(X\times^{\log}X\times^{\log}Y_1\times^{\log}...\times^{\log}Y_n\times^{\log}Z)\to D(X\times^{\log}Y_1\times^{\log}...\times^{\log}Y_n)\]
    and
    \[\tilde{q}_*\tilde{j}^*:D(X\times^{\log}X\times^{\log}Y_1\times^{\log}...\times^{\log}Y_n\times^{\log}Z)\to D(X\times^{\log}Y_1\times^{\log}...\times^{\log}Y_n)\]
    are equivalent.

    Furthermore, the functors
    \[\tilde{p}^*\tilde{i}_*:D(X\times^{\log}Y_1\times^{\log}...\times^{\log}Y_n)\to D(X\times^{\log}X\times^{\log}Y_1\times^{\log}...\times^{\log}Y_n\times^{\log}Z)\]
    and
    \[\tilde{j}_*\tilde{q}^*: D(X\times^{\log}Y_1\times^{\log}...\times^{\log}Y_n)\to D(X\times^{\log}X\times^{\log}Y_1\times^{\log}...\times^{\log}Y_n\times^{\log}Z)\]
    are also equivalent.
\end{lemma}

\begin{proof}
    We prove this lemma following our usual three steps.
    \begin{itemize}
        \item We show that the underived fiber product, $B$ and the derived fiber product, $B^R$, agree.
        \item We show that $B$ has mild singularities (in this case, $B$ will be smooth).
        \item We show that the natural map $\pi:X\times^{\log}Y_1\times^{\log}...\times^{\log}Y_n\times^{\log}Z\to B$ is a birational map, i.e the functor $\pi_*\pi^*$ is equivalent to the identity functor.
    \end{itemize}

    The first step follows from the fact that the map $\tilde{p}$ is flat (see Lemma \ref{lem:flatprojection}).
    
    We compute the second step explicitly. We show that $B$ is smooth. In order to do so, we work \'etale locally. We assume that all the spaces are affine spaces and the boundary divisors are cut out by regular functions $x_1, x_2, y_1,...,y_n,z,$ on the two copies of $X$, $Y_1,..., Y_n$ and $Z$ respectively. As before, we omit the remaining smooth coordinates, since they only contribute polynomial-ring factors.
    
    Consider a chart of
    \[X\times^{\log}X\times^{\log}Y_1\times^{\log}...\times^{\log}Y_n\times^{\log}Z.\]
    Let $d_1,...,d_{n+3}$ be the sequence obtained by ordering the regular functions $x_1,x_2,y_1,...,y_n,z$ according to the corresponding permutation. Recall that the chart has coordinates $t_1, ..., t_{n+3}$ so that 
    \[d_k=t_k t_{k+1}...t_{n+3}.\]
    The diagonal map $X\to X\times^{\log}X$ forces on the fiber product to have $t_rt_{r+1}...t_{s-1}=1$ for some $r<s$. This equation is smooth, showing that $B$ itself is smooth.
    
    Finally, the third step follows from the fact that on the open locus (outside of the divisors) the diagram is Cartesian, therefore the map $\pi:X\times^{\log}Y_1\times^{\log}...\times^{\log}Y_n\times^{\log}Z\to B$ is a birational map. Since $B$ is smooth (and $\pi$ is proper), we get that $\pi_*\pi^*$ is equivalent to the identity functor on perfect complexes. This implies our statement via a sequence of equivalence as in the proof of Lemma \ref{lem:gensquare} or in the proof of Theorem \ref{thm:generalsquare}.
\end{proof}


Using induction, this lemma can be generalized to obtain the following statement. 

\begin{theorem}\label{thm:generalsquare2}
    Let $(X,D_X)$, $(Y_1, D_1), ..., (Y_n, D_n)$, $(Z_1, D_1'), ..., (Z_m, D_m')$ be log smooth pairs. Then, consider the diagram
    \[\Tiny
    \begin{tikzcd}  X\times^{\log}Y_1\times^{\log}...\times^{\log}Y_n\times^{\log}Z_1\times^{\log}...\times^{\log}Z_m\ar[r, "\tilde{j}"]\ar[d, "\tilde{q}"] & X\times^{\log} X\times^{\log}Y_1\times^{\log}...\times^{\log}Y_n\times^{\log}Z_1\times^{\log}...\times^{\log}Z_m\ar[d, "\tilde{p}"]\\
    X\times^{\log}Y_1\times^{\log}...\times^{\log}Y_n \ar[r, "\tilde{i}"] & X\times^{\log} X\times^{\log}Y_1\times^{\log}...\times^{\log}Y_n
    \end{tikzcd}\]
    obtained from the log diagonal maps and projections. Then, we have that the functors
    \[\tilde{i}^*\tilde{p}_*\quad \mbox{and}\quad\tilde{q}_*\tilde{j}^*\]
    are equivalent.

    Furthermore, the functors
    \[\tilde{p}^*\tilde{i}_*\quad \mbox{and}\quad \tilde{j}_*\tilde{q}^*\]
    are also equivalent.
\end{theorem}

The proof is completely parallel to the proof of Theorem \ref{thm:generalsquare}, hence the proof is omitted.

\section{Log Fourier-Mukai kernels}
\label{sec:logFM}

In this section, we use the world of blow-ups defined in the previous section to define logarithmic versions of Fourier-Mukai kernels. Since there does not exist a universally accepted "derived category of log sheaves" we will focus on log Fourier-Mukai kernels rather than on log Fourier-Mukai transforms. 

Given a smooth and proper scheme $X$, throughout the paper we denote a dg-enhancement of the derived category of coherent sheaves by $D(X)$. Let $E\in D(X\times^{\log}Y)$ be a perfect complex on $X\times^{\log} Y$. We treat it as a kernel and we define the corresponding logarithmic ``Fourier-Mukai'' functor
\[
\phi_E : D(X) \to D(Y); \qquad \phi_E(F)=\pi_{Y, *}(\pi_X^*F\otimes E).
\]
The maps $\pi_X$ and $\pi_Y$ are the blowdown composed with the natural projection maps and are proper:
\[
\begin{tikzcd}
    & Bl_{D_X\times D_Y}(X\times Y)=X\times^{\log}Y\ar[ld, "\pi_X", swap]\ar[rd, "\pi_Y"] & \\
    X & & Y.
\end{tikzcd}
\]

Write $\rho : X \times^{\log} Y \times^{\log} Z \to X \times^{\log} Z$ and $\pi_{AB}$ for the morphism from Lemma \ref{lem:projection} out of $X \times^{\log} Y \times^{\log} Z$ lying over the projection onto $A, B \in \{X, Y, Z\}$.
\[
\begin{tikzcd}
     & & X\times^{\log}Y\times^{\log}Z \ar[rd, "\pi_{YZ}"]\ar[ld, "\pi_{XY}"]\ar[rr, "\rho"]& &X\times^{\log}Z\\
     & X\times^{\log}Y\ar[rd]\ar[ld] & &    Y\times^{\log}Z\ar[rd]\ar[ld] & \\
    X &  & Y & & Z
\end{tikzcd}
\]

\begin{definition}
    Given $E\in D(X\times^{\log}Y)$ and $F\in D(Y\times^{\log}Z)$, the composite $\phi_F \circ \phi_E$ is the log Fourier-Mukai transform associated to the log Fourier-Mukai kernel $\rho_*(E \boxtimes F) = \rho_*(\pi_{XY}^* E \otimes \pi_{YZ}^* F)$:
\begin{align*}
    (\phi_F \circ \phi_E) (M) &= \pi_{Z*}(F \otimes \pi_Y^* (\pi_{Y*} (E \otimes \pi_X^* M)))         \\
            &=\pi_{Z*}(\rho_* (E \boxtimes F) \otimes \pi_X^* M).
\end{align*}
\end{definition}



\begin{remark}
    We note that the kernel of the composite log Fourier-Mukai transformation is also perfect. Indeed, the map $\rho$ is proper, $X\times^{\log}Z$ is smooth, hence $\rho_*$ sends perfect complexes to perfect complexes. 
\end{remark}

Again, since there is no universally accepted notion of a derived category of log sheaves, we prefer to work with kernels and not with transforms.


\begin{lemma}\label{lem:associativity}
    The log Fourier-Mukai transformation is associative: given kernels $E_{XY}\in D(X\times^{\log}Y)$, $E_{YZ}\in D(Y\times^{\log}Z)$ and $E_{ZW}\in D(Z\times^{\log}W)$, the kernels of the composite log Fourier-Mukai transformations $\phi_{E_{ZW}}\circ (\phi_{E_{YZ}} \circ\phi_{E_{XY}})$ and of $(\phi_{E_{ZW}}\circ \phi_{E_{YZ}}) \circ\phi_{E_{XY}}$ are canonically isomorphic.
\end{lemma}

\begin{proof}
    Consider $E_{XY}\in D(X\times^{\log}Y)$, $E_{YZ}\in D(Y\times^{\log}Z)$ and $E_{ZW}\in D(Z\times^{\log}W)$ as log Fourier-Mukai kernels of the log Fourier-Mukai transformations. We compute the log Fourier-Mukai kernel corresponding to $\phi_{E_{ZW}}\circ (\phi_{E_{YZ}} \circ\phi_{E_{XY}})$. For this, we use the diagram
    \[
    \begin{tikzcd}
        X\times^{\log}Y\times^{\log}Z\times^{\log}W\ar[r, "\pi_{XZW}"]\ar[d, "\pi_{XYZ}"] & X\times^{\log}Z\times^{\log}W\ar[d, "\pi_{XZ}"]\\
        X\times^{\log}Y\times^{\log}Z\ar[r, "\rho"]& X\times^{\log}Z
    \end{tikzcd}.
    \]
    The log Fourier-Mukai kernel corresponding to $\phi_{E_{ZW}}\circ (\phi_{E_{YZ}} \circ\phi_{E_{XY}})$ is given by
    \[\pi_{XW, *}(\pi_{XZ}^*\rho_*(\pi_{XY}^*E_{XY}\otimes \pi_{YZ}^*E_{YZ})\otimes \pi_{ZW}^*E_{ZW})\]
    where $\pi_{XW}:X\times^{\log}Z\times^{\log}W\to X\times^{\log}W$ is the natural projection map given by Lemma \ref{lem:projection}. Using Theorem \ref{thm:generalsquare}, we see that $\pi^*_{XZ}\rho_*\simeq \pi_{XZW, *}\pi_{XYZ}^*$. Therefore, the log Fourier-Mukai kernel of $\phi_{E_{ZW}}\circ (\phi_{E_{YZ}} \circ\phi_{E_{XY}})$ can be represented by the kernel 
    \[\rho'_*(E_{XY}\boxtimes E_{YZ}\boxtimes E_{ZW})\]
    where $\rho'$ is the natural map $X\times^{\log}Y\times^{\log}Z\times^{\log}W\to X\times^{\log}W$ and $E_{XY}\boxtimes E_{YZ}\boxtimes E_{ZW}$ is defined as the tensor product of the pullbacks of the kernels via the maps $X\times^{\log}Y\times^{\log}Z\times^{\log}W\to X\times^{\log}Y$, $X\times^{\log}Y\times^{\log}Z\times^{\log}W\to Y\times^{\log}Z$ and $X\times^{\log}Y\times^{\log}Z\times^{\log}W\to Z\times^{\log}W$.

    The kernel of the log Fourier-Mukai transformation $(\phi_{E_{ZW}}\circ \phi_{E_{YZ}}) \circ\phi_{E_{XY}}$ can be computed similarly providing the same answer, $\rho'_*(E_{XY}\boxtimes E_{YZ}\boxtimes E_{ZW})$. This proves our statement.
\end{proof}

\begin{remark}
    Similarly, let $E_1\in D(X_1\times^{\log}X_2)$, ..., $E_n\in D(X_n\times^{\log}X_{n+1})$ be log Fourier-Mukai kernels. Then, their composite
    \[\phi_{E_n}\circ...\circ \phi_{E_1}\]
    is given by the log Fourier-Mukai kernel
    \[\nu_*(E_1\boxtimes E_2\boxtimes...\boxtimes E_n)\]
    where $\nu:X_1\times^{\log}...\times^{\log}X_{n+1}\to X_1\times^{\log}X_{n+1}$ is the projection given by Lemma \ref{lem:projection} and 
    \[E_1\boxtimes...\boxtimes E_n\]
    is defined as the tensor product of the pullbacks of the kernels $E_i$ via the maps $X_1\times^{\log}...\times^{\log}X_{n+1}\to X_i\times^{\log}X_{i+1}$.

    As in the lemma above, this is independent of the order of the operations taken in the composition.
\end{remark}

Furthermore, the log Fourier-Mukai transformation is unital, meaning that the log Fourier-Mukai transforms have left and right units given by the log diagonal embeddings. For that we need the following technical lemma.

Now, we are ready to show that log Fourier-Mukai transformation has left and right units.

\begin{lemma}\label{lem:unit}
    The log Fourier-Mukai transformation has left and right units. In other words, given a log Fourier-Mukai transformation $\phi_E:D(X)\to D(Y)$, the log Fourier-Mukai transformations $\phi_{i_{Y,*}\OO_Y}\circ \phi_E$, $\phi_E$ and $\phi_E\circ \phi_{i_{X,*}\OO_X}$ are isomorphic (meaning that the corresponding kernels are canonically isomorphic). Here $i_X$ and $i_Y$ denote the log diagonal embeddings $X\to X\times^{\log}X$ and $Y\to Y\times^{\log}Y$ respectively.
\end{lemma}

\begin{proof}
    We only do the proof for the right unit (the proof for the left unit is completely similar). In other words, we show that the kernels of $\phi_E\circ \phi_{i_*\OO_X}$ and $\phi_E$ are isomorphic. The other proof is completely similar.

    We consider the following diagram.
     \[
    \begin{tikzcd}
      & & X\times^{\log}X\times^{\log}Y\ar[rd, "e"]\ar[ld, "p"]\ar[rr, "\rho"]& & X\times^{\log}Y\\
     & X\times^{\log}X\ar[rd]\ar[ld] & &    X\times^{\log}Y\ar[rd]\ar[ld] & \\
    X &  & X & & Y.
    \end{tikzcd}
    \]
    The kernel of $\phi_E\circ \phi_{i_*\OO_X}$ is represented by the kernel 
    \[\rho_*(e^*E\otimes p^*i_*\OO_X)\] where $\rho$ denotes the map $\rho:X\times^{\log}X\times^{\log}Y\to X\times^{\log}Y$ that is the projection on the first and third components.
    
    Using the diagram
    \[
\begin{tikzcd}
    X\times^{\log} Y\ar[r, "j"]\ar[d, "q"] & X\times^{\log}X\times^{\log} Y\ar[d, "p"]\\
    X\ar[r, "i"] & X\times^{\log} X.
\end{tikzcd} \]
    by Lemma \ref{lem:square2}, the functors $p^*i_*$ and $j_*q^*$ are equivalent, hence the kernel is given by $\rho_*(e^*E\otimes j_*\OO_{X\times^{\log}Y})$ on $X\times^{\log}Y$. Using the projection formula for $j_*$, we obtain that the kernel is $\rho_*j_*(j^*e^*E)$. Since, $e\circ j=\id$ and $\rho\circ j=\id$, we obtain that the kernel of $\phi_E\circ \phi_{i_*\OO_X}$ is given by $E$ on $X\times^{\log}Y$. This proves our claim.    
\end{proof}

\subsection{Non-categorical left and right adjoints}

In this section, we construct non-categorical unit and counit maps. These unit and counit maps behave similarly to the usual unit and counit maps, but they do not come from an underlying category (hence the name ``non-categorical"). First, we take a bit of a detour and discuss logarithmic Serre functors.

\subsubsection{Log Serre functors}
In this section, we prove a couple of 
lemmas on the log Serre functor.

\begin{definition}
    Let $X$ be a log smooth log scheme. The \emph{log dualizing sheaf} is the top wedge power of its log K\"ahler differentials, $\wedge^{\dim X} \lkah{X}$. The logarithmic Serre functor (\cite{leonardi2025logarithmic}) is the tensor product with the twisted sheaf $-\otimes S_X^{\log}$ that is obtained by shifting the log dualizing sheaf
    \[S_X^{\log}=\wedge^{\dim X} \lkah{X} [\dim X].\] 
\end{definition}

\begin{remark}
    We believe that in a suitable category of log sheaves, this complex behaves as a Serre functor.
\end{remark}


The log Serre functor is related to the relative Serre functor of the log diagonal.

\begin{lemma}\label{lem:relativeSerreinverselogSerredual}
Let $(X,D)$ be a smooth log pair. Then, the relative Serre functor of the log diagonal embedding $i:X\to X\times^{\log}X$ can be identified with the inverse of the log Serre functor:
\[S_{X/X\times^{\log}X}=\omega_X^{-1}(-D)[-\dim X]=(S_X^{\log})^{-1}.\]
\end{lemma}

\begin{proof}
    It is shown in \cite{HABLICSEK2026127, leonardi2025logarithmic} that if $X$ is a quasicompact, weakly log separated, log smooth log scheme, then the relative Serre of the log diagonal embedding $\delta:X\to X\times_{\af{X}}X$ can be identified with the inverse of the log Serre functor:
    \[S_{X/X\times_{\af{X}}X}=(S_X^{\log})^{-1}.\] 
    Furthermore, in the special case of a log smooth pair $(X,D)$, the variety $X\times_{\af{X}}X$ is an open subvariety of $X\times^{\log}X$ (see \cite{HABLICSEK2026127}). This implies the statement.
\end{proof}

We relate the log Serre functor to log Hochschild homology.

\begin{lemma}\label{lem:serre}
    Let $X$ be a quasicompact, weakly log separated, log smooth log scheme such as a smooth, proper log pair $X = (X, D)$. 
    We have an isomorphism of dg vector spaces
    \[R\Gamma(X, \HHl{X}(\OO_X))\simeq \RHom_{X\times_{\af{X}}X}(\delta_*\OO_X, \delta_*S_X^{\log}).\]

    Furthermore, if $(X,D)$ is a log smooth pair, then we have an isomorphism of dg vector spaces
     \[R\Gamma(X, \HHl{X}(\OO_X))\simeq \RHom_{X\times^{\log}X}(i_*\OO_X, i_*S_X^{\log}).\]
\end{lemma}

\begin{proof}
    The proofs of the two statements are very similar. We only show the second statement. Using adjunction, we have
    \[ \RHom_{X\times^{\log}X}(i_*\OO_X, i_*S_X^{\log})\simeq \RHom_{X}(\OO_X, i^!i_*S_X^{\log}).\]

    The formality theorem (\cite{HABLICSEK2026127}) says there exists an equivalence of dg endofunctors
    \[i^!i_*(-)\simeq (-)\otimes \Sym(T_X^{\log}[-1])\]
    to get that 
    \[i^!i_*(S_X^{\log})\simeq S_X^{\log}\otimes\Sym(T_X^{\log}[-1])\simeq \Sym(\Omega^{1,\log}_X[1]).\]
    Here, the last equality comes from the isomorphisms of vector bundles
    \[\Omega^{\dim X, \log}_X\otimes \wedge^q T_X^{\log}\simeq \Omega^{\dim X-q, \log}_X\]
    or equivalently a quasi-isomorphism of shifted vector bundles
    \[S_X^{\log}\otimes \wedge^q T_X^{\log}[-q]\simeq \Omega_X^{\dim X-q, \log}[\dim X-q].\]
    Taking derived global sections, we obtain
    \[
    \RHom_{X}(\OO_X,i^!i_*S_X^{\log})\simeq R\Gamma\!\left(X,\Sym(\Omega^{1,\log}_X[1])\right),
    \] 
    which agrees with $R\Gamma(X,\HHl{X}(\OO_X))$ by the logarithmic HKR theorem. This proves the claim.
\end{proof}

The statement needs some further explanations. 


\begin{remark}\label{rmk:ordinaryHHbackwardsSerreduality}
    The analogue of Lemma \ref{lem:serre} for non-logarithmic Hochschild homology can be proven using Serre duality. 
    Let the diagonal map $X\to X\times X$ be denoted by $\Delta$. We have the following sequence of isomorphisms of dg vector spaces
    \[\RHom_{X\times X}(\Delta_*\OO_X, \Delta_*S_X)\simeq \RHom_{X}(\Delta^*\Delta_*\OO_X, S_X)\simeq \]
    \[\simeq\RHom_{X}(\OO_X, \Delta^*\Delta_*\OO_X)^{\vee}=\HH(X)^\vee.\]
    Here the first isomorphism comes from adjunction, and the second from Serre duality. 

    Note that there is an extra dual in the sequence above on Hochschild homology. However, Serre duality again provides isomorphisms of vector spaces $\HH_n(X)^{\vee}\simeq \HH_{-n}(X)$, and thus, the dual can be omitted with a similar proof as above.
    
    Due to the lack of a Serre duality for logarithmic Hochschild homology, we cannot use the sequence of isomorphisms above. This is a crucial point, and this is the reason why our functor in Theorem \ref{thm:main} is contravariant.
\end{remark}

\subsubsection{Non-categorical right adjoint}
Consider $\phi_E:D(X)\to D(Y)$ and another log FM kernel $\phi_{\tilde{E}}:D(Y)\to D(X)$, and their composite $\phi_{\tilde{E}}\circ \phi_E:D(X)\to D(X)$ that is represented by the kernel $E\boxtimes \tilde{E}$ on $X\times^{\log}Y\times^{\log} X$. We have the following commutative diagram of ordinary schemes (without log structure) as in Lemma \ref{lem:square2}:
\begin{equation}
\begin{tikzcd}
    X\times^{\log} Y\ar[r, "j"]\ar[d, "q"] & X\times^{\log}Y\times^{\log} X\ar[d, "p"]\\
    X\ar[r, "i"] & X\times^{\log} X
\end{tikzcd} \end{equation}



Using the diagram above, we obtain the non-categorical unit map.

\begin{lemma}\label{lem:right_adj}
    We obtain a map of complexes
    \[i_*\OO_X\to p_*(E\boxtimes \tau(E^{\vee}\otimes q^*S^{\log}_{X}))\]
    that is induced by the co-evaluation map where $\tau$ is the map $\tau:X\times^{\log}Y\to Y\times^{\log}X$ that flips the coordinates.
\end{lemma}

\begin{proof}
    Using the usual adjunction type formula $i^!(-)\simeq i^*(-)\otimes S_{X/X\times^{\log}X}$, we have that we need to construct a map
    \[\OO_X\to i^!p_*(E\boxtimes \tau(E^{\vee}\otimes q^*S^{\log}_{X}))\simeq i^*p_*(E\boxtimes \tau(E^{\vee}\otimes q^*S^{\log}_{X}))\otimes S_{X/X\times^{\log}X}.\]
    From Lemma \ref{lem:square2}, we have an equivalence of functors $i^*p_*\simeq q_*j^*$, and hence a quasi-isomorphism of complexes
    \[i^*p_*(E\boxtimes \tau(E^{\vee}\otimes q^*S^{\log}_{X}))\simeq q_*j^*(E\boxtimes \tau(E^{\vee}\otimes q^*S^{\log}_{X})).\]
    Thus, we need to construct a map
    \[\OO_X\to q_*j^*(E\boxtimes \tau(E^{\vee}\otimes q^*S^{\log}_{X}))\otimes S_{X/X\times^{\log}X}.\]
    Using projection formula, the latter complex is quasi-isomorphic to \[q_*(j^*(E\boxtimes \tau(E^{\vee}\otimes q^*S^{\log}_{X}))\otimes q^*S_{X/X\times^{\log}X}).\] Using adjunction again, therefore, we need to construct a map
    \[\OO_{X\times^{\log}Y}\to j^*(E\boxtimes \tau(E^{\vee}\otimes q^*S^{\log}_{X}))\otimes q^*S_{X/X\times^{\log}X}.\]
    The relative Serre functor $S_{X/X\times^{\log}X}$ can be described explicitly as $(S_X^{\log})^{-1}$ (see \cite{HABLICSEK2026127}). Furthermore, the composite maps 
    \[X\times^{\log}Y\xrightarrow{j}X\times^{\log}Y\times^{\log}X\xrightarrow{\pi_{XY}} X\times^{\log}Y\]
    and
    \[X\times^{\log}Y\xrightarrow{j}X\times^{\log}Y\times^{\log}X\xrightarrow{\pi_{YX}} Y\times^{\log}X\xrightarrow{\tau}X\times^{\log}Y\]
    are the identity map. The Serre functors cancel out by Lemma \ref{lem:relativeSerreinverselogSerredual}:
    \[q^*S_X^{\log}\otimes q^*S_{X/X\times^{\log}X}=\OO_{X\times^{\log}Y}.\] 
    So we only need to construct a map
    \[\OO_{X\times^{\log}Y}\to j^*(E\boxtimes \tau E^\vee)=E\otimes E^\vee.\]
    Taking the co-evaluation map, our statement is proven.
\end{proof}

\begin{definition}
    We denote $\phi_{\tau(E^{\vee}\otimes q^*S_X^{\log})}$ as $\phi_E^{!}$, and call it the \emph{non-categorical right adjoint} of the log Fourier-Mukai transform $\phi_E$.
\end{definition}

This is not an actual adjoint of $\phi_E$. However, we have a unit-like natural transformation $\id_X\to \phi_E^!\circ \phi_E$ given by the lemma above. We think that this is an actual right adjoint in the correct category of sheaves over a log scheme.

\begin{remark}
    While the non-categorical right adjoint always exists, it does not have desirable properties like functoriality. The reason may be the discrepancy between the derived category of coherent sheaves on $X$ and the potential derived category of log sheaves on $X$; and this is why we work with strong log Fourier-Mukai kernels.
\end{remark}

To circumvent the problem, we restrict ourselves to certain kind of Fourier-Mukai transforms.

\subsubsection{Strong log Fourier-Mukai kernels}

\begin{definition}
    We say that $E$ is a \textit{strong} log Fourier-Mukai kernel on $X\times^{\log}Y$ if $E$ is perfect and $E$ is supported outside of the strict transforms $\tilde{D_X\times Y}$, $\tilde{X\times D_Y}$ of $D_X\times Y$ and $X\times D_Y$ respectively, meaning that the derived pullback vanishes on the strict transforms
    \[E\otimes^L\OO_{\tilde{D_X\times Y}}=0=E\otimes^L\OO_{\tilde{X\times D_Y}}.\]
    
    We call the corresponding log Fourier-Mukai transform a \textit{strong} log Fourier-Mukai transform.
\end{definition}

In the rest of the paper, we usually omit the derived notation $\otimes^{L}$ (and only use $\otimes$), but we wanted to emphasize here that strongness means that the derived restriction of the kernels vanish on the strict transforms.

\begin{remark}
    The definition above has a very natural interpretation via Artin fans. In fact, being a strong log Fourier-Mukai kernel in $X\times^{\log}Y$ is equivalent to requiring that the kernel is supported on $X\times_{[\Aff^1/\GG_m]}Y$.
\end{remark}

\begin{remark}
    The definition of a strong log Fourier-Mukai kernel was motivated by the definition of a static sheaf \cite{dell2026coherent} that was communicated to us by Patrick Kennedy-Hunt.
\end{remark}

\begin{example}
    Let $f:X\to Y$ be a map of smooth log pairs, i.e a map of smooth varieties inducing an equality of Cartier divisors $D_X=f^*D_Y$. Then, the graph $\OO_{\Gamma_f}\in D(X\times^{\log} Y)$ of $f$ is a strong log Fourier-Mukai kernel that induces the strong log Fourier-Mukai transform $f_*:D(X)\to D(Y)$. Similarly, its transpose $\tau \OO_{\Gamma_f} \in D(Y\times^{\log}X)$ induces the strong log Fourier-Mukai transform $f^*:D(Y)\to D(X)$.
\end{example}

We show that composition of strong log Fourier-Mukai transforms is also a strong log Fourier-Mukai transform.

\begin{lemma}\label{lem:strongFMcomp}
    Let $\phi_E:D(X)\to D(Y)$ and $\phi_F:D(Y)\to D(Z)$ be strong log Fourier-Mukai transforms. Then, $\phi_F\circ \phi_E:D(X)\to D(Z)$ is also a strong log Fourier-Mukai transform. 
\end{lemma}

\begin{proof}
        Consider the following diagram
    \begin{equation}\label{dia:xyz}
\begin{tikzcd}
     & & X\times^{\log}Y\times^{\log}Z \ar[rd, "w"]\ar[ld, "v"]\ar[rr, "u"]& &X\times^{\log}Z\\
     & X\times^{\log}Y\ar[rd]\ar[ld] & &    Y\times^{\log}Z\ar[rd]\ar[ld] & \\
    X &  & Y & & Z
\end{tikzcd}
    \end{equation}
    The kernel of the strong log Fourier-Mukai transform is given by
    \[u_*(E\boxtimes F)=u_*(v^*E\otimes w^*F).\]
    We need to show that this kernel is supported outside of the strict transforms of $\tilde{D_X\times Z}$ and $\tilde{X\times D_Z}$ inside the blow-up $X\times^{\log}Z$. These strict transforms are disjoint, and without loss of generality it is enough to show that 
    \[u_*(v^*E\otimes w^*F)\otimes \OO_{\tilde{D_X\times Z}}=0.\]
    We factor the map $u$ as the composite 
    \[X\times^{\log}Y\times^{\log}Z\to X\times^{\log}Z\times Y\to X\times^{\log}Z\]
    where the first map is a sequential blow-up map considered in Example \ref{ex:blowup3} and the second map is the projection map. Using the explicit description of Example \ref{ex:blowup3}, the map $X\times^{\log}Y\times^{\log}Z\to X\times^{\log}Z\times Y$ is a sequential blow-up obtained by blowing up the dominant transforms of the subvarieties given by first $D_X\times D_Y\times D_Z$, then of $X\times D_Y\times D_Z$ and finally, of $D_X\times D_Y\times Z$. This shows that 
    \[u^{-1}(\tilde{D_X\times Z})\coloneq (\tilde{D_X\times Z})\times_{X\times^{\log}Z}(X\times^{\log}Y\times^{\log}Z)\] is the union of 
    the exceptional divisor of the blow-up along the strict transform of $D_X\times D_Y\times Z$, and the strict transform of $D_X\times Y\times Z$. The exceptional divisor of the blow-up corresponding to $D_X\times D_Y\times D_Z$ does not appear in $u^{-1}(\tilde{D_X\times Z})$ as it maps to the exceptional divisor of $X\times^{\log}Z$, and not to the strict transform of $D_X\times Z$.
    
    Since $\tilde{D_X\times Z}$ is an effective Cartier divisor on a smooth variety and $u$ is flat (see Lemma \ref{lem:flatprojection}), we have that $u^{-1}(\tilde{D_X\times Z})$ is also an effective Cartier divisor (on the smooth variety $X\times^{\log}Y\times^{\log}Z$). Furthermore, since $u$ is flat,  $u^{-1}(\tilde{D_X\times Z})$ is also the derived fiber product
    \[(\tilde{D_X\times Z})\times^R_{X\times^{\log}Z}(X\times^{\log}Y\times^{\log}Z).\]
    Thus, in order to show that 
    \[u_*(v^*E\otimes w^*F)\otimes \OO_{\tilde{D_X\times Z}}=0,\]
    we need to show that
    \[(v^*E\otimes w^*F)|_{u^{-1}(\tilde{D_X\times Z})}=0.\]
    From the explicit description of $u^{-1}(\tilde{D_X\times Z})$, we see that
    \[u^{-1}(\tilde{D_X\times Z})\subset v^{-1}(\tilde{D_X\times Y})\cup w^{-1}(\tilde{D_Y\times Z}).\]
    Furthermore, since $E$ and $F$ are strong log Fourier-Mukai kernels, we have that 
    \[v^*E|_{v^{-1}(\tilde{D_X\times Y})}=0=w^*F|_{w^{-1}(\tilde{D_Y\times Z})}.\]
    This shows that 
    \[(v^*E\otimes w^*F)|_{u^{-1}(\tilde{D_X\times Z})}=0\]
    finishing our proof.
\end{proof}

\begin{remark}
    The statement above is motivated by the following. If the kernel $E$ is supported on $X\times_{[\Aff^1/\GG_m]}Y$ and $F$ is supported on $Y\times_{[\Aff^1/\GG_m]}Z$, then $v^*E\otimes w^*F$ is supported on 
    \[X\times_{[\Aff^1/\GG_m]}Y\times_Y Y\times_{[\Aff^1/\GG_m]}Z=X\times_{[\Aff^1/\GG_m]}Y\times_{[\Aff^1/\GG_m]}Z.\] 
    This is the open locus of $X\times^{\log}Y\times^{\log}Z$ outside of the strict transforms of the boundary divisors $D_X\times Y\times Z$, $X\times D_Y\times Z$ and $X\times Y\times D_Z$.
\end{remark}

In the following lemma, we show that the right adjoint is functorial for strong Mukai transforms.

\begin{lemma}\label{lem:right_adj_func}
    Let $\phi_E$ be a strong log Fourier-Mukai transform $D(X)\to D(Y)$ and $\phi_F$ a strong log Fourier-Mukai transform $D(Y)\to D(Z)$. Then there exists a canonical quasi-isomorphism from the kernel of the strong log Fourier-Mukai transform $\phi^!_{F\circ E}$ to the kernel of the strong log Fourier-Mukai transform $\phi^!_E\circ \phi^!_F$.
\end{lemma}

\begin{proof}
    We again consider Diagram \ref{dia:xyz}.
    \[
    \begin{tikzcd}
     & & X\times^{\log}Y\times^{\log}Z \ar[rd, "w"]\ar[ld, "v"]\ar[rr, "u"]& &X\times^{\log}Z\\
     & X\times^{\log}Y\ar[rd]\ar[ld] & &    Y\times^{\log}Z\ar[rd]\ar[ld] & \\
    X &  & Y & & Z
    \end{tikzcd}\]
    
We compute the kernels of the strong log Fourier-Mukai transforms $\phi^!_{(F\circ E)}$ and $\phi_{E}^!\circ \phi_{F}^!$ on $X\times^{\log}Y\times^{\log}Z$ (meaning that we look at the kernel before taking the pushforward $u_*$).

The kernel of $\phi_{(F\circ E)}^!$ is given by 
$(u_*(v^*E\otimes w^*F))^{\vee}\otimes \tilde{q}^*S_X^{\log}$ where $\tilde{q}$ is the projection map $X\times^{\log}Z\to X$. Since $u:X\times^{\log}Y\times^{\log}Z\to X\times^{\log}Z$ is a proper morphism between smooth varieties, Grothendieck duality gives a canonical isomorphism
\[\mathcal{Hom}_{X\times^{\log}Z}(u_*(v^*E\otimes w^*F), \OO_{X\times^{\log}Z})\simeq u_*\mathcal{Hom}_{X\times^{\log}Y\times^{\log}Z}(v^*E\otimes w^*F, u^!\OO_{X\times^{\log}Z}).\]
Therefore, we get a canonical quasi-isomorphism of complexes
\[(u_*(v^*E\otimes w^*F))^{\vee}\simeq u_*(v^*E^\vee\otimes w^*F^\vee\otimes S_{X\times^{\log}Y\times^{\log}Z/X\times^{\log}Z}),\]
and thus, we get
\[(u_*(v^*E\otimes w^*F))^{\vee}\otimes \tilde{q}^*S_X^{\log}\simeq u_*(v^*E^\vee\otimes w^*F^\vee\otimes S_{X\times^{\log}Y\times^{\log}Z/X\times^{\log}Z})\otimes \tilde{q}^*S_X^{\log}.\]

The kernel of $\phi_{E}^!\circ \phi_{F}^!$ is given by $u_*(v^*E^\vee\otimes w^*F^\vee\otimes v^*q^*S_X^{\log}\otimes w^*\tilde{t}^*S_Y^{\log})$ where $\tilde{t}$ is the projection map $Y\times^{\log}Z\to Y$.

Therefore, in order to show our lemma, we need to show that the two kernels are isomorphic. This means that we need to show that the two complexes
\[v^*E^\vee\otimes w^*F^\vee\otimes S_{X\times^{\log}Y\times^{\log}Z/X\times^{\log}Z}\otimes u^*\tilde{q}^*S_X^{\log}\]
and
\[v^*E^\vee\otimes w^*F^\vee\otimes v^*q^*S_X^{\log}\otimes w^*\tilde{t}^*S_Y^{\log}\]
are quasi-isomorphic.

First note that $u^*\tilde{q}^*S_X^{\log}$ and $v^*q^*S_X^{\log}$ are canonically quasi-isomorphic, since the diagram
\[
\begin{tikzcd}
    X\times^{\log}Y\times^{\log}Z\ar[r, "u"]\ar[d, "v"] & X\times^{\log}Z\ar[d, "\tilde{q}"]\\
    X\times^{\log}Y\ar[r, "q"] & X
\end{tikzcd}\]
commutes.

We now investigate the discrepancy between the relative Serre functor
\[
S_{X\times^{\log}Y\times^{\log}Z/X\times^{\log}Z}
\]
and the pullback $w^*\tilde{t}^*S_Y^{\log}$. Since the difference in dimension between $X\times^{\log}Y\times^{\log}Z$ and $X\times^{\log}Z$ is the dimension of $Y$, we see that \[S_{X\times^{\log}Y\times^{\log}Z/X\times^{\log}Z} \otimes (w^*\tilde{t}^*S_Y^{\log})^{-1}\] is quasi-isomorphic to a line bundle concentrated in degree 0 that we will describe now using Example \ref{ex:blowup3}. 

We factor the map
    \[u:X\times^{\log}Y\times^{\log}Z\to X\times^{\log}Z\]
    as 
     \[X\times^{\log}Y\times^{\log}Z\to X\times^{\log}Z\times Y\to X\times^{\log}Z.\]
    Since $X\times^{\log}Y\times^{\log}Z\to X\times^{\log}Z\times Y$ is obtained by sequentially blowing up the dominant transforms of $D_X\times D_Y\times D_Z$, of $D_X\times D_Y\times Z$ and finally of $X\times D_Y\times D_Z$, the shifted line bundle $S_{X\times^{\log}Y\times^{\log}Z/X\times^{\log}Z}$ is twisted by the divisor coming from the exceptional divisor lying over $D_X\times D_Y\times D_Z$, the exceptional divisor lying over the dominant transform of $D_X\times D_Y\times Z$ and the exceptional divisor lying over the dominant transform of $X\times D_Y\times D_Z$. These exceptional divisors also occur in the transform of $X\times^{\log}Z\times D_Y$. Therefore:
    \[S_{X\times^{\log}Y\times^{\log}Z/X\times^{\log}Z}\otimes (w^*\tilde{t}^*S_Y^{\log})^{-1}\simeq \OO(-\tilde{D_Y})\]
    where $\tilde{D_Y}$ is the dominant transform of $X\times^{\log}Z\times D_Y$.


The strict transform $\tilde{D_Y}$ is supported on the union
\[v^{-1}(\tilde{D_X\times Y})\cup v^{-1}(\tilde{X\times D_Y})\cup w^{-1}(\tilde{D_Y\times Z})\cup w^{-1}(\tilde{Y\times D_Z}).\]
In fact, $\tilde{D_Y}$ maps to the strict transform of $X\times D_Y$ in $X\times^{\log}Y$. Therefore, since $E$ and $F$ are strong log Fourier-Mukai kernels, we have that
\[v^*E^\vee\otimes w^*F^\vee\otimes \OO_{\tilde{D_Y}}=0.\]

Furthermore, we get a map
\[v^*E^\vee\otimes w^*F^\vee\otimes S_{X\times^{\log}Y\times^{\log}Z/X\times^{\log}Z}\otimes u^*\tilde{q}^*S_X^{\log}\to v^*E^\vee\otimes w^*F^\vee\otimes v^*q^*S_X^{\log}\otimes w^*\tilde{t}^*S_Y^{\log}\]
induced by the inclusion
\[\OO_{X\times^{\log}Y\times^{\log}Z}(-\tilde{D_Y})\to  \OO_{X\times^{\log}Y\times^{\log}Z}\].
We show that this map is a quasi-isomorphism.




For that we use the exact triangle 
\[\OO_{X\times^{\log}Y\times^{\log}Z}(-\tilde{D_Y})\to  \OO_{X\times^{\log}Y\times^{\log}Z}\to  \OO_{\tilde{D_Y}}\to^{+1}.\]
We tensor the exact triangle above with 
\[v^*E^\vee\otimes w^*F^\vee\otimes v^*q^*S_X^{\log}\otimes w^*\tilde{t}^*S_Y^{\log}\] to get that the map 
\[v^*E^\vee\otimes w^*F^\vee\otimes S_{X\times^{\log}Y\times^{\log}Z/X\times^{\log}Z}\otimes u^*\tilde{q}^*S_X^{\log}\to v^*E^\vee\otimes w^*F^\vee\otimes v^*q^*S_X^{\log}\otimes w^*\tilde{t}^*S_Y^{\log}\]
is a canonical quasi-isomorphism. Thus, there exists a canonical quasi-isomorphism from the kernel of the strong log Fourier-Mukai transform $\phi^!_{F\circ E}$ to the kernel of the strong log Fourier-Mukai transform $\phi^!_E\circ \phi^!_F$
showing our claim.
\end{proof}

\begin{remark}
    Lemma \ref{lem:right_adj_func} is motivated by the fact that the relative canonical bundle of 
    \[\beta:X\times_{[\Aff^1/\GG_m]}Y\times_{[\Aff^1/\GG_m]}Z\to X\times_{[\Aff^1/\GG_m]}Z\] is isomorphic to the log canonical bundle of $Y$ pulled back via the composite map \[\gamma:X\times_{[\Aff^1/\GG_m]}Y\times_{[\Aff^1/\GG_m]}Z\to X\times Y\times Z\to Y.\]
    In other words, 
    \[S_{X\times_{[\Aff^1/\GG_m]}Y\times_{[\Aff^1/\GG_m]}Z/ X\times_{[\Aff^1/\GG_m]}Z}\otimes \gamma^*(S_Y^{\log})^{-1}=\OO_{X\times_{[\Aff^1/\GG_m]}Y\times_{[\Aff^1/\GG_m]}Z}.\]
    In some sense, Lemma \ref{lem:right_adj_func} is a way of adapting the above computation to a compactified setting.
\end{remark}



\subsubsection{Non-categorical left adjoint}

Now, we are ready to define the non-categorical left adjoint. Consider again the commutative diagram \ref{dia:cart}
  \[\begin{tikzcd}
    X\times^{\log} Y\ar[r, "j"]\ar[d, "q"] & X\times^{\log}Y\times^{\log} X\ar[d, "p"]\\
    X\ar[r, "i"] & X\times^{\log} X.
\end{tikzcd}\]

\begin{lemma}\label{lem:left_adj}
    Let $E$ be a strong log Fourier-Mukai kernel on $X\times^{\log}Y$ providing a strong log Fourier-Mukai transform $\phi_E:D(X)\to D(Y)$. We have a natural map of complexes 
    \[p_*(E\boxtimes \tau(E^\vee\otimes t^*S_Y^{\log}))\to i_*\OO_X.\] 
    Here, $t:X\times^{\log}Y\to Y$ is the projection map and $\tau$ denotes the map $\tau:X\times^{\log}Y\to Y\times^{\log}X$ that flips the coordinates as well as the induced pushforward on the derived category.
\end{lemma}

\begin{proof}
    Using Lemma \ref{lem:square2}, we have an equivalence of functors $i^*p_*\simeq q_*j^*$. Thus, we have a quasi-isomorphism of complexes
    \[i^*p_*(E\boxtimes \tau(E^{\vee}\otimes t^*S_Y^{\log}))\simeq q_*j^*(E\boxtimes \tau(E^{\vee}\otimes t^*S_Y^{\log})).\]
    Since, the composite maps 
    \[X\times^{\log}Y\xrightarrow{j}X\times^{\log}Y\times^{\log}X\xrightarrow{\pi_{XY}} X\times^{\log}Y\]
    and
    \[X\times^{\log}Y\xrightarrow{j}X\times^{\log}Y\times^{\log}X\xrightarrow{\pi_{YX}} Y\times^{\log}X\xrightarrow{\tau}X\times^{\log}Y\]
    are the identity map, we have that 
    \[q_*j^*(E\boxtimes \tau(E^{\vee}\otimes t^*S_Y^{\log}))\simeq q_*(E\otimes E^\vee\otimes t^*S_Y^{\log}).\]
    Therefore, we need to construct a map
    \[q_*(E\otimes E^{\vee}\otimes t^*S_Y^{\log})\to \OO_X.\]
    In order to do so, we construct a map 
    \[E\otimes E^{\vee}\otimes t^*S_Y^{\log}\to q^!\OO_X=S_{X\times^{\log}Y/X}.\]
    The discrepancy between $t^*S_Y^{\log}$ and $S_{X\times^{\log}Y/X}$ is given by the strict transform $\tilde{X\times D_Y}$ of $X\times D_Y$ inside $X\times^{\log}Y$, namely
    \[S_{X\times^{\log}Y/X}\otimes (t^*S_Y^{\log})^{-1}=\OO_{X\times^{\log}Y}(-\tilde{X\times D_Y}).\]
    Since $E$ is a strong Fourier-Mukai kernel, hence it is supported outside of this strict transform. This implies that $E\otimes E^{\vee}\otimes \OO_{\tilde{X\times D_Y}}=0$. Applying $(E\otimes E^{\vee})\otimes-$ to the exact triangle
    \[\OO_{X\times^{\log}Y}(-\tilde{X\times D_Y})\to \OO_{X\times^{\log}Y}\to \OO_{\tilde{X\times D_Y}}\to^{+1}\]
    we obtain
    \[E\otimes E^{\vee}\otimes \OO_{X\times^{\log}Y}(-\tilde{X\times D_Y})\to E\otimes E^{\vee}\otimes \OO_{X\times^{\log}Y}\to 0\to^{+1}\]
    showing that the evaluation map $E\otimes E^{\vee}\to \OO_{X\times^{\log}Y}$ lifts uniquely to
    \[E\otimes E^{\vee}\to \OO_{X\times^{\log}Y}(-\tilde{X\times D_Y}).\]
    This provides a map $E\otimes E^{\vee}\otimes t^*S_Y^{\log}\to q^!\OO_X$, and thus a map 
    \[p_*(E\boxtimes \tau(E^\vee\otimes t^*S_Y^{\log}))\to i_*\OO_X\] finishing our proof.
\end{proof}

\begin{remark}
    The lemma above has a natural interpretation via Artin fans. In fact, the relative Serre functor of $S_{X\times_{[\Aff^1/\GG_m]}Y/X}$ is exactly $t^*S_Y^{\log}$ coming from the fact that the following Cartesian square is a derived Cartesian square (since $X$ and $Y$ are log flat).
    \[
    \begin{tikzcd}
        X\times_{[\Aff^1/\GG_m]}Y\ar[r]\ar[d] & X\ar[d]\\
        Y\ar[r] & {[\Aff^1/\GG_m]}
    \end{tikzcd}
    \]
\end{remark}

\begin{definition}
   We denote $\phi_{\tau(E^{\vee}\otimes t^*S_Y^{\log})}$ as $\phi_{E}^*$, and this is what we call the non-categorical left adjoint.
\end{definition}

\begin{remark}
    Lemma \ref{lem:right_adj} and \ref{lem:left_adj} provide non-categorical unit and counit morphisms. 
\end{remark}


Similar to Lemma \ref{lem:right_adj_func} we show that the non-categorical left adjoint is also functorial.

\begin{lemma}\label{lem:left_adj_func}
        Let $\phi_E$ be a strong log Fourier-Mukai transform $D(X)\to D(Y)$ and $\phi_F$ a strong log Fourier-Mukai transform $D(Y)\to D(Z)$. Then, there exists a canonical quasi-isomorphism from the kernel of the strong log Fourier-Mukai transform $\phi_{(F\circ E)}^*$ to the kernel of the strong log Fourier-Mukai transform $\phi_{E}^*\circ \phi_{F}^*$.
\end{lemma}

\begin{proof}
    The proof is very similar to Lemma \ref{lem:right_adj_func}, we highlight the main steps. We consider Diagram \ref{dia:xyz}. The kernel of $\phi^*_{F\circ E}$ is given by $(u_*(v^*E\otimes w^*F))^\vee\otimes s^*S_Z^{\log}$ where $s$ is the projection map $X\times^{\log}Z\to Z$. Using Grothendieck duality, we can identify this kernel as
    \[u_*(v^*E^\vee\otimes w^*F^\vee\otimes S_{X\times^{\log}Y\times^{\log}Z/X\times^{\log}Z})\otimes s^*S_Z^{\log}.\]
    The kernel of $\phi^*_E\circ \phi^*_F$, on the other hand, is given by
    \[u_*(v^*E^\vee\otimes w^*F^\vee\otimes v^*t^*S_Y^{\log}\otimes w^*\tilde{s}^*S_Z^{\log})\]
    where $t:X\times^{\log}Y\to Y$ and $\tilde{s}:Y\times^{\log}Z\to Z$ are the projection maps. 

    Using that the diagram
    \[
    \begin{tikzcd}
        X\times^{\log}Y\times^{\log}Z\ar[r, "u"]\ar[d, "w"]& X\times^{\log}Z\ar[d, "s"]\\
        Y\times^{\log}Z\ar[r, "\tilde{s}"]& Z
    \end{tikzcd}
    \]
    is commutative, we obtain that $u^*s^*S_Z^{\log}$ and $w^*\tilde{s}^*S_Z^{\log}$ are canonically isomorphic.

    Now, we compare $S_{X\times^{\log}Y\times^{\log}Z/X\times^{\log}Z}$ and $v^*t^*S_Y^{\log}$. Using that the diagram
      \[
    \begin{tikzcd}
        X\times^{\log}Y\times^{\log}Z\ar[r, "w"]\ar[d, "v"]& Y\times^{\log}Z\ar[d, "\tilde{t}"]\\
        X\times^{\log}Y\ar[r, "t"]& Y
    \end{tikzcd}
    \]
    is commutative, we have that $v^*t^*S_Y^{\log}$ and $w^*\tilde{t}^*S_Y^{\log}$ can be identified. This means that the line bundle (concentrated in degree 0)
    \[S_{X\times^{\log}Y\times^{\log}Z/X\times^{\log}Z}\otimes (v^*t^*S_Y^{\log})^{-1}\]
    is the same line bundle ($L=\OO(-\tilde{D_Y})$) appearing in the proof of Lemma \ref{lem:right_adj_func}. Therefore, we obtain a map from the kernel of the strong log Fourier-Mukai transform $\phi^*_{F\circ E}$ to the kernel of the strong log Fourier-Mukai transform $\phi^*_E\circ \phi^*_F$.
    
    To show that this map is a quasi-isomorphism, we use the proof of Lemma \ref{lem:right_adj_func} verbatim. Namely, using again that $E$ and $F$ are strong log Fourier-Mukai kernels, we get that 
    \[v^*E^\vee\otimes w^*F^\vee\otimes \OO_{\tilde{D_Y}}=0,\]
    and thus tensoring the exact triangle 
    \[\OO(-\tilde{D_Y})\to \OO\to \OO_{\tilde{D_Y}}\xrightarrow{+1}\]
    with $v^*E^\vee\otimes w^*F^\vee\otimes v^*t^*S_Y^{\log}\otimes w^*\tilde{s}^*S_Z^{\log}$, we get that the map
    \[u_*(v^*E^\vee\otimes w^*F^\vee\otimes S_{X\times^{\log}Y\times^{\log}Z/X\times^{\log}Z}\otimes u^*s^*S_Z^{\log})\to\]
    \[\to u_*(v^*E^\vee\otimes w^*F^\vee\otimes v^*t^*S_Y^{\log}\otimes w^*\tilde{s}^*S_Z^{\log})\]
    is a canonical quasi-isomorphism. Thus, there exists a canonical quasi-isomorphism from the kernel of the strong log Fourier-Mukai transform $\phi_{(F\circ E)}^*$ to the kernel of the strong log Fourier-Mukai transform $\phi_{E}^*\circ \phi_{F}^*$ showing our claim.
\end{proof}

Furthermore, we have the following relation between the non-categorical left and right adjoints.

\begin{lemma}\label{lem:adj}
    Let $E$ be a strong log Fourier-Mukai kernel inducing a strong log Fourier-Mukai transform $\phi_E:D(X)\to D(Y)$. The adjunctions induce a canonical isomorphism of kernels of the log Fourier-Mukai transforms
    \[\phi_E^! \circ \phi_{\tilde{i}_*(S_{Y}^{\log})}\quad\mbox{and}\quad \phi_{i_*S_X^{\log}}\circ\phi_E^*\]
    where $\tilde{i}$ is the log diagonal $Y\to Y\times^{\log}Y$.
\end{lemma}

\begin{proof}
    Let us compute the kernel corresponding to the log Fourier-Mukai transformation of the functor on the left. For this, we consider the following diagram
    \[
\begin{tikzcd}
     & & Y\times^{\log}Y\times^{\log}X\ar[rd, "d"]\ar[ld, "c"]& &\\
     & Y\times^{\log}Y\ar[rd]\ar[ld] & &    Y\times^{\log}X\ar[rd]\ar[ld] & \\
    Y &  & Y & & X.
\end{tikzcd}
\]
The kernel of the functor of the left can be represented by the kernel on $Y\times^{\log}Y\times^{\log}X$ given by
\[c^*\tilde{i}_*S_Y^{\log}\otimes d^*E^{\vee}\otimes d^*p^*S_X^{\log}.\]
Now, we have a commutative diagram similar to Lemma \ref{lem:square2}
\[
\begin{tikzcd}
    Y\times^{\log} X\ar[r, "\tilde{j}"]\ar[d, "t"] & Y\times^{\log}Y\times^{\log} X\ar[d, "c"]\\
    Y\ar[r, "\tilde{i}"] & Y\times^{\log} Y.
\end{tikzcd}
\]
Similarly to Lemma \ref{lem:square2}, we have that
\[c^*\tilde{i}_*S_Y^{\log}\simeq \tilde{j}_*t^*S_Y^{\log}\]
and thus
\[c^*\tilde{i}_*S_Y^{\log}\otimes d^*E^{\vee}\otimes d^*p^*S_X^{\log}\simeq 
\tilde{j}_*t^*S_Y^{\log}\otimes d^*E^{\vee}\otimes d^*p^*S_X^{\log}
.\]
Using projection formula, we get that
\[\tilde{j}_*t^*S_Y^{\log}\otimes d^*E^{\vee}\otimes d^*p^*S_X^{\log}\simeq 
\tilde{j}_*(t^*S_Y^{\log}\otimes E^{\vee}\otimes p^*S_X^{\log}).
\]
This implies that the kernel of the log Fourier Mukai transform of the functor of the left is given by 
\[t^*S_Y^{\log}\otimes E^{\vee}\otimes p^*S_X^{\log}.\]
Similar computation can be done to the other side as well that agrees with the kernel above.
\end{proof}

\begin{remark}
    Lemma \ref{lem:adj} explains why the proofs of Lemmas \ref{lem:right_adj_func} and \ref{lem:left_adj_func} are quite similar: there is a direct way to go from the non-categorical right adjoints to the non-categorical left adjoints. Explicitly, the non-categorical right adjoint and the non-categorical left adjoint agree up to twisting with the logarithmic Serre functors.
\end{remark}

\section{Functoriality of log Hochschild homology}
\label{sec:functoriality}
In this section, we show that strong log Fourier-Mukai
kernels induce contravariant functorial maps on logarithmic
Hochschild homology. Our argument is inspired by \cite{caldararu2003mukai} and the construction is analogous to the construction in \cite{caldararu2003mukai}.

For a strong log Fourier-Mukai kernel $E\in D(X\times^{\log}Y)$ thought of as a strong log Fourier-Mukai transform $\phi_E:D(X)\to D(Y)$, we consider the sequence of strong log Fourier-Mukai transforms
\[\Psi_{E,\beta}:\phi_{i_{X,*}\OO_X}\Rightarrow  \phi_E^!\circ\phi_E\xRightarrow{\beta}  \phi_E^! \circ \phi_{i_{Y,*}(S_{Y}^{\log})}\circ\phi_E\Rightarrow \phi_{i_{X,*}S_X^{\log}}\circ\phi_E^*\circ \phi_E\Rightarrow  \phi_{i_{X,*}S_X^{\log}}.\]

We emphasize that in this sequence of natural transformations of strong log Fourier-Mukai transforms every step is induced by a map of strong log Fourier-Mukai kernels. In fact, the first map is given by the non-categorical unit map. The second map is given by a log Hochschild element $\beta\in \RHom_{Y\times^{\log}Y}(i_{Y,*}\OO_Y, i_{Y,*}S_Y^{\log})$. The third map is given by the adjunction formula of Lemma \ref{lem:adj}. Finally, the last map is the non-categorical counit map.

\begin{lemma}\label{lem:hhfunctoriality}
For any element $\beta\in R\Gamma(Y, \HHl{Y}(\OO_Y))$ in the log Hochschild homology of $Y$, we have that $\Psi_{E,\beta}$ is an element of the dg vector space $\RHom_{X\times^{\log}X}(i_*\OO_X, i_*S_X^{\log})$. In other words, $\Psi_{E,\beta}$ is an element in the log Hochschild homology of $X$, $R\Gamma(X, \HHl{X}(\OO_X))$.
\end{lemma}
\begin{proof}
The maps of Fourier-Mukai transformations in the definition of $\Psi_{E,\beta}$ are all induced by maps of kernels on $X\times^{\log}X$. Therefore, we obtain a composite map of kernels that is a map $i_{X,*}\OO_X\to i_{X,*}S_X^{\log}$. The composite map thus gives us an element as claimed. 
\end{proof}

\begin{remark}
    Since all maps included above (unit, counit, adjoint-comparison) are linear maps, the assignment $\beta\mapsto \Psi_{E, \beta}$ is also a linear map.
\end{remark}

\begin{remark}
    Furthermore, all maps included above are closed morphisms of degree 0, therefore,
    \[d\Psi_{E,\beta}=\Psi_{E, d\beta}.\]
    This means that Lemma \ref{lem:hhfunctoriality} induces a linear map of vector spaces
    \[\Psi_{E}:\HH^{\log}_n(Y)\to \HH_n^{\log}(X)\]
    for every $n$.
\end{remark}

In order to show functoriality, we need to strengthen Lemmas \ref{lem:right_adj_func}, \ref{lem:left_adj_func} and \ref{lem:adj} in the following two lemmas to show that the functoriality of the adjunction is coherent with the unit transformations. 

\begin{lemma}\label{lem:adjunctioncoherence}
Let $E\in D(X\times^{\log}Y)$ and $F\in D(Y\times^{\log}Z)$ be strong log Fourier-Mukai kernels. Under the canonical isomorphisms $(\phi_{F}\circ\phi_E)^!\simeq \phi_E^!\circ\phi_F^!$ and $(\phi_F\circ\phi_E)^*\simeq \phi_E^*\circ \phi_F^*$ constructed in Lemmas \ref{lem:right_adj_func} and \ref{lem:left_adj_func} we have the following statements.
\begin{itemize}
\item The unit map $\phi_{i_*\OO_X}\Rightarrow (\phi_F\circ \phi_E)^!\circ (\phi_F\circ \phi_E)$ is canonically equivalent to the composite of unit maps
\[\phi_{i_*\OO_X}\Rightarrow \phi_E^!\circ \phi_E\Rightarrow \phi_E^!\circ \phi_F^!\circ \phi_F\circ \phi_E.\]

\item The counit map $(\phi_F\circ \phi_E)^*\circ (\phi_F\circ \phi_E)\Rightarrow \phi_{i_*\OO_X}$ is canonically equivalent to the composite of counit maps
\[\phi_E^*\circ \phi_F^*\circ \phi_F\circ \phi_E\Rightarrow \phi^*_E\circ \phi_E\Rightarrow \phi_{i_*\OO_X}.\]
\end{itemize}
\end{lemma}

\begin{proof}
    We only prove the first statement, the second statement is fairly similar.

Consider the commutative diagram induced by the log diagonal map $X\to X\times^{\log}X$ and projection maps (see Lemma \ref{lem:projection}).
\[\Small
\begin{tikzcd}[column sep=large,row sep=large]
X\times^{\log}Y\times^{\log}Z
 \ar[r,"e"]\ar[d,"a"]
&
X\times^{\log}Y\times^{\log}Z\times^{\log}Y
 \ar[r,"f"]\ar[d,"b"]
&
X\times^{\log}Y\times^{\log}Z
 \times^{\log}Y\times^{\log}X
 \ar[d,"c"]
\\
X\times^{\log}Y
 \ar[r,"g"]
&
X\times^{\log}Y\times^{\log}Y
 \ar[r,"h"]\ar[d,"r"]
&
X\times^{\log}Y\times^{\log}Y\times^{\log}X
 \ar[d,"d"]
\\
{}
&
X\ar[r,"i_X"]
&
X\times^{\log}X .
\end{tikzcd}
\]

The kernel of $(\phi_F\circ \phi_E)^!\circ (\phi_F\circ \phi_E)$ is given by (see Lemma \ref{lem:right_adj_func})
    \[(d\circ c)_*((E\boxtimes F)\boxtimes \tau (F^\vee\boxtimes E^\vee\boxtimes S_X^{\log}\otimes S_{X\times^{\log}Y\times^{\log}Z/X\times^{\log}Z})) \]
    while the kernel of $\phi_E^!\circ \phi_F^!\circ \phi_F\circ \phi_E$ is given by
    \[(d\circ c)_*(E\boxtimes F\boxtimes \tau(F^\vee\otimes S_Y^{\log})\boxtimes \tau(E^\vee\otimes S_X^{\log})).\]
    We describe now the morphism 
    \[i_*\OO_X\to (d\circ c)_*((E\boxtimes F)\boxtimes \tau (F^\vee\boxtimes E^\vee\boxtimes S_X^{\log}\otimes S_{X\times^{\log}Y\times^{\log}Z/X\times^{\log}Z}))\]
    that induces the unit map
    \[\phi_{i_*\OO_X}\Rightarrow (\phi_F\circ \phi_E)^!\circ (\phi_F\circ \phi_E)\]
    From Theorem \ref{thm:generalsquare2}, we get that the functors $d^*i_{X,*}$ and $h_*r^*$ are equivalent, hence using the adjunction, we need to describe a morphism
    \[h_*\OO_{X\times^{\log}Y\times^{\log}Y}\to c_*((E\boxtimes F)\boxtimes \tau (F^\vee\boxtimes E^\vee\boxtimes S_X^{\log}\otimes S_{X\times^{\log}Y\times^{\log}Z/X\times^{\log}Z})).\]
    This is given as follows. First, we use the unit map
    \[\OO_{X\times^{\log}Y\times^{\log}Y}\to g_*\OO_{X\times^{\log}Y}.\]
    Second we construct a map
    \[(h\circ g)_*\OO_{X\times^{\log}Y}\to c_*((E\boxtimes F)\boxtimes \tau (F^\vee\boxtimes E^\vee\boxtimes S_X^{\log}\otimes S_{X\times^{\log}Y\times^{\log}Z/X\times^{\log}Z})).\]
    We use again Theorem \ref{thm:generalsquare2} to show that the functors $c^*(h\circ g)_*$ and $(f\circ e)_*a^*$ are equivalent, so it is enough to construct a map
    \[(f\circ e)_*\OO_{X\times^{\log}Y\times^{\log}Z}\to (E\boxtimes F)\boxtimes \tau (F^\vee\boxtimes E^\vee\boxtimes S_X^{\log}\otimes S_{X\times^{\log}Y\times^{\log}Z/X\times^{\log}Z})\]
    in other words a map
    \[\OO_{X\times^{\log}Y\times^{\log}Z}\to (f\circ e)^!( (E\boxtimes F)\boxtimes \tau (F^\vee\boxtimes E^\vee\boxtimes S_X^{\log}\otimes S_{X\times^{\log}Y\times^{\log}Z/X\times^{\log}Z})).\]
    Using that $E$ and $F$ are strong kernels, we see that this map is induced by the co-evaluation map 
    \[\OO_{X\times^{\log}Y\times^{\log}Z}\to (E\boxtimes F)\otimes (E\boxtimes F)^\vee\]
    on $X\times^{\log}Y\times^{\log}Z$. The same computation can be done  for the other kernel as well showing that it is induced by the composite of co-evaluation maps
    \[\OO_{X\times^{\log}Y\times^{\log}Z}\to (E\otimes E^\vee)\boxtimes (F\otimes F^\vee).\]
    These are canonically equivalent and the equivalence is induced by the map used in Lemma \ref{lem:right_adj_func} implying our statement.
\end{proof}

For a strong log Fourier--Mukai kernel
\(E\in D(X\times^{\log}Y)\), let
\[
\vartheta_E:
\phi_E^!\circ\phi_{i_{Y,*}S_Y^{\log}}
\xrightarrow{\sim}
\phi_{i_{X,*}S_X^{\log}}\circ\phi_E^*
\]
denote the canonical isomorphism constructed in
Lemma \ref{lem:adj}.

\begin{lemma}\label{lem:adjcoherence}
Let $E\in D(X\times^{\log}Y)$ and $F\in D(Y\times^{\log}Z)$ be strong log Fourier--Mukai kernels. Under the canonical isomorphisms $(\phi_F\circ\phi_E)^!\simeq \phi_E^!\circ\phi_F^!$ and $(\phi_F\circ\phi_E)^*\simeq\phi_E^*\circ\phi_F^*$ constructed in Lemmas \ref{lem:right_adj_func} and \ref{lem:left_adj_func}, the following diagram commutes
\[
\begin{tikzcd}[column sep=large,row sep=large]
(\phi_F\circ\phi_E)^!
 \circ\phi_{i_{Z,*}S_Z^{\log}}
 \ar[rr,"\vartheta_{F\circ E}"]
 \ar[d,"\sim"']
&&
\phi_{i_{X,*}S_X^{\log}}
 \circ(\phi_F\circ\phi_E)^*
 \ar[d,"\sim"]
\\
\phi_E^!\circ\phi_F^!
 \circ\phi_{i_{Z,*}S_Z^{\log}}
 \ar[r,"\id_{\phi_E^!}\circ\vartheta_F"']
&
\phi_E^!\circ
 \phi_{i_{Y,*}S_Y^{\log}}\circ\phi_F^*
 \ar[r,"\vartheta_E\circ\id_{\phi_F^*}"']
&
\phi_{i_{X,*}S_X^{\log}}
 \circ\phi_E^*\circ\phi_F^*.
\end{tikzcd}
\]
\end{lemma}

\begin{proof}
We understand the maps of this diagram. After computing the kernels explicitly, we see that the horizontal maps are the identity maps (see Lemma \ref{lem:adj}). The vertical parts of the diagram are induced by the same discrepancy morphism (see Lemmas \ref{lem:right_adj_func} and \ref{lem:left_adj_func})
\[S_{X\times^{\log}Y\times^{\log}Z/X\times^{\log}Z}\to p_Y^*S_Y^{\log}\]
hence the diagram commutes (here $p_Y:X\times^{\log}Y\times^{\log}Z\to Y$ is the projection map).
\end{proof}

Using the lemmas above, we obtain the main result of the paper.

\begin{theorem}\label{thm:main}
    Let $\phi_E:D(X)\to D(Y)$ be a strong log Fourier-Mukai transform. Then, we have an induced contravariant (functorial) map on log Hochschild homology
    \begin{center}
    \begin{tabular}{r c c l}
    $\phi^{HH}_E: $& $R\Gamma(Y,\HHl{Y}(\OO_Y))$ & $\to$ & $R\Gamma(X, \HHl{X}(\OO_X))$\\
    & $\beta$ & $\mapsto$ & $\Psi_{E,\beta}$.
    \end{tabular}
    \end{center}
\end{theorem}

\begin{proof}
    The functoriality follows from the functoriality of the non-categorical left and right adjoints, see Lemmas \ref{lem:left_adj_func} and \ref{lem:right_adj_func}. Explicitly, let $\phi_E:D(X)\to D(Y)$ and $\phi_F:D(Y)\to D(Z)$ be two strong log Fourier-Mukai transforms. Then, the induced map on log Hochschild homology $R\Gamma(Z,\HHl{Z}(\OO_Z))\to R\Gamma(X, \HHl{X}(\OO_X))$ of $(\phi_F\circ \phi_E)$ is given as the composite map
    \[\phi_{i_{X,*}\OO_X}\Rightarrow  (\phi_F\circ \phi_E)^!\circ(\phi_F\circ\phi_E)\Rightarrow  (\phi_F\circ\phi_E)^! \circ \phi_{i_{Z,*}(S_{Z}^{\log})}\circ(\phi_F\circ\phi_E)\Rightarrow\]
    \[\Rightarrow \phi_{i_{X,*}S_X^{\log}}\circ(\phi_F\circ\phi_E)^*\circ (\phi_F\circ\phi_E)\Rightarrow  \phi_{i_{X,*}S_X^{\log}}.\]
    On the other hand, the induced map on log Hochschild homology \[R\Gamma(Z,\HHl{Z}(\OO_Z))\to R\Gamma(X, \HHl{X}(\OO_X))\] of the composite $\phi^{HH}_E\circ \phi^{HH}_F$ is given as the composite map
    \[\phi_{i_{X,*}\OO_X}\Rightarrow  \phi_E^!\circ \left(\phi_F^!\circ\phi_F\right)\circ\phi_E\Rightarrow  \phi_E^! \circ\left(\phi_F^!\circ \phi_{i_{Z,*}(S_{Z}^{\log})}\circ \phi_F\right)\circ\phi_E\Rightarrow\]
    \[\Rightarrow \phi_E^! \circ \phi_{i_{Y,*}S_Y^{\log}}\circ \phi_F^*\circ \phi_F\circ \phi_E\Rightarrow \phi_{i_{X,*}S_X^{\log}}\circ \phi_E^*\circ \phi_F^*\circ \phi_F\circ \phi_E\Rightarrow \phi_{i_{X,*}S_X^{\log}}.\]
    The kernels corresponding to the Fourier-Mukai transforms of
    $(\phi_F\circ \phi_E)^!\circ (\phi_F\circ \phi_E)$ and $\phi_E^!\circ \phi_F^!\circ \phi_F\circ \phi_E$ are quasi-isomorphic (see Lemma \ref{lem:right_adj_func}), furthermore, the kernels corresponding to
    $(\phi_F\circ\phi_E)^*\circ (\phi_F\circ\phi_E)$ and $\phi_E^*\circ \phi_F^*\circ \phi_F\circ \phi_E$ are quasi-isomorphic (see Lemma \ref{lem:left_adj_func}). Lastly, the non-categorical unit and counit maps (see Lemmas \ref{lem:right_adj} and \ref{lem:left_adj}) are induced by the corresponding coevaluation and evaluation morphisms for $E$ and $F$. Therefore, we have a commutative diagram of kernels from the diagram of kernels corresponding to $\phi_{F\circ E}$ to the diagram of kernels corresponding to $\phi_F\circ \phi_E$. The two kernel-level constructions are identified term by term, and using Lemmas \ref{lem:adjunctioncoherence} and \ref{lem:adjcoherence} we see that all squares commute. Therefore, the corresponding maps on log Hochschild homology are equal. 
\end{proof}

\begin{example}
    Since log Hochschild homology is functorial with respect to the strong log Fourier-Mukai transforms, we get that $\phi_{i_{X,*}\OO_X}$ acts as the identity on the log Hochschild homology of $X$.
\end{example}

\begin{remark}
    Note that in Theorem \ref{thm:main} strong log Fourier-Mukai transforms induce a contravariant map on log Hochschild homology as opposed to the non-log setting when Fourier-Mukai transforms induce a covariant map on Hochschild homology \cite{caldararu2003mukai, kuznetsov2009}. This is an artifact of the absence of a Poincar\'e/Serre duality in the logarithmic setting as in Remark \ref{rmk:ordinaryHHbackwardsSerreduality}.    
\end{remark}

\section{Applications}
\label{sec:applications}

\subsection{A log derived bicategory}\label{sec:logcat}
We can put our framework into a categorical framework and define a dg-bicategory of smooth and proper log pairs. Our construction is similar to the construction of the bicategory in the case of smooth and proper varieties (see \cite{caldararu2010mukai}). 

\begin{definition}
    We define the dg-bicategory of smooth and proper log pairs $\LogPair$ as follows. Objects of this bicategory are log smooth pairs $(X,D)$ with $X$ being smooth and proper and $D$ being a smooth geometrically connected Cartier divisor. The 1-morphisms $(X,D_X)\to (Y,D_Y)$ form a dg category \[
    \Hom_{\LogPair}((X,D_X),(Y,D_Y))
    \]
    which is the full dg-subcategory of $D(X\times^{\log}Y)$ spanned by strong log Fourier-Mukai kernels. The identity 1-morphism $(X,D_X)\to (X,D_X)$ is given by the strong log Fourier-Mukai kernel $i_*\OO_X\in D(X\times^{\log}X)$. If $E,F\in D(X\times^{\log}Y)$ are strong kernels, then 2-morphisms $E\to F$ are morphisms in the dg-category $D(X\times^{\log}Y)$, i.e $\Hom_{D(X\times^{\log}Y)}(E, F)$.

    The composition of 1-morphisms is given by the composition of strong log Fourier-Mukai kernels
    \[
    F\circ E\coloneq u_*(v^*E\otimes w^*F)\in D(X\times^{\log}Z)
    \]
    (see Lemmas \ref{lem:associativity} and \ref{lem:strongFMcomp}).

    The vertical composition of 2-morphisms is given by composition of morphisms in $D(X\times^{\log}Y)$. Lastly, if $\alpha:E\to E'$ and $\beta:F\to F'$ are 2-morphisms, their horizontal composition is defined by the corresponding map on strong log Fourier-Mukai kernels
    \[
    \beta\star\alpha \coloneq  \rho_*(v^*\alpha\otimes w^*\beta):
    \rho_*(v^*E\otimes w^*F)\to \rho_*(v^*E'\otimes w^*F').
    \]
\end{definition}

\begin{theorem}
\label{thm:dgbicat}
    The compositions defined above make $\LogPair$ a dg-bicategory.
\end{theorem}

\begin{proof}
    The identity 1-morphism $D(X)\to D(X)$ is given by $i_*\OO_X\in D(X\times^{\log}X)$ which is a unit up to a canonical isomorphism (see Lemma \ref{lem:unit}). The composition of 1-morphisms is given by composition of strong Fourier-Mukai kernels (Lemmas \ref{lem:associativity} and \ref{lem:strongFMcomp}) which is associative up to a canonical isomorphism. The vertical composition of 2-morphisms is given by composition of morphisms between strong log Fourier–Mukai kernels. The horizontal composition of 2-morphisms is given by the induced map on the composition of 1-morphisms. Finally, the compatibility between the horizontal and vertical composition of 2-morphisms follows from the explicit form of the kernel of the composition of 1-morphisms (see Lemma \ref{lem:associativity}) in $D(X\times^{\log}Y\times^{\log}Z)$. Namely given vertical homogeneous 2-morphisms $\alpha:E\to E'$, $\alpha':E'\to E''$ and $\beta:F\to F'$, $\beta':F'\to F''$, the map on the kernels corresponding to $(\beta'\circ \beta)\star(\alpha'\circ \alpha)$ and to $(-1)^{|\beta'||\alpha|}(\beta'\star \alpha')\circ (\beta\star \alpha)$ agree (with the Koszul sign), namely given by the map
    \[\rho_*(v^*(\alpha'\circ\alpha)\otimes w^*(\beta'\circ\beta)):\rho_*(v^*E\otimes w^*F)\to \rho_*(v^*E''\otimes w^*F'').\]
    Furthermore, using Lemmas \ref{lem:associativity} and \ref{lem:unit}, the unitors and associators are given by the identity morphism between kernels on a certain log product. For instance, for four composable
    kernels $E_i\in D(X_{i-1}\times^{\log} X_i)$ ($1\leq i\leq 4)$, all five kernels appearing in the pentagon are represented by the same kernel that is the pushforward of
    \[E_1\boxtimes E_2\boxtimes E_3\boxtimes E_4\]
    on the log product
    \[X_0\times^{\log}X_1\times^{\log}X_2\times^{\log}X_3\times^{\log}X_4.\]
    Furthermore, the associators constructed in Lemma \ref{lem:associativity} are all induced by the identity morphism, and thus the pentagon commutes. Similar argument holds for the triangle identity as also. 
    
    This shows our claim.
\end{proof}

A remarkable 
observation is that log Hochschild homology and cohomology are invariants of this framework. Indeed, Theorem \ref{thm:main} shows that log Hochschild homology is functorial (contravariantly) for strong logarithmic kernels and is invariant under equivalences in this bicategory. Furthermore, we have that the dg endomorphism algebra of the kernel $i_*\OO_X$ of the identity functor $D(X)\to D(X)$ (as a strong log Fourier-Mukai transform) can be identified with log Hochschild cohomology (similar to \cite{caldararu2003mukai, ramadoss2010mukai}).

The dg endomorphism algebra is given by
\[\Hom_{D(X\times^{\log}X)}(i_*\OO_X,i_*\OO_X)\]
as in \cite{toen2007homotopy}. This agrees, by definition, with the dg vector space $R\Gamma(X, \cHHl{X}(\OO_X))$ of log Hochschild cohomology of $X$.

\begin{remark}
    Our dg-bicategory can be thought of as a naive version of a category of log schemes where the objects are log schemes and morphisms are given by certain log correspondences. The discussion suggests that it may be easier to define the derived category of log sheaves not on a single log scheme $X$, but rather on products of the form $X\times^{\log}Y$ as correspondences.
\end{remark}

\subsection{Log Chern character}
\label{sec:logChern}

Let $(X,D)$ be a smooth log pair. Currently, there is no widely accepted notion of log coherent sheaves. In this section, we will define the log Chern character for specific logarithmic correspondences. Our definition of log Chern character has log Hochschild homology as a target contrary to other definitions in the literature (see, for instance, \cite{scherotzke2018additive}).

We consider two main cases.

\subsubsection{Log Chern character of perfect complexes of $X$}\label{sec:chernperf}
Recall that $X\times^{\log}X$ is defined as $\mathrm{Bl}_{D \times D}(X \times X)$. Consider the diagram
\[
\begin{tikzcd}
    & X\times^{\log}X\ar[dl, "\pi_1"]\ar[dr, "\pi_2"] &\\
    (X, D) & & (X,D).
\end{tikzcd}
\]
 As above, let $i: X \to X\times^{\log}X$ be the log diagonal map whose composite with the map $X\times^{\log}X\to X\times X$ (from the blow-up to the base) gives the embedding of the diagonal $\Delta:X \to X\times X$.
We consider a perfect complex $E$ on $X$, and regard it as a strong Fourier-Mukai kernel $i_*E\in D(X\times^{\log}X)$. Note that $i_*E$ is indeed a strong Fourier-Mukai kernel as it is supported on the log diagonal. 

Using Theorem \ref{thm:main}, we obtain a corresponding linear map on log Hochschild homology
\[\phi^{HH}_{i_*E}:R\Gamma(X, \HHl{X}(\OO_X))\to R\Gamma(X, \HHl{X}(\OO_X)).\]
Using the logarithmic version of HKR, we have
\[\HH^{\log}_*(X)=R^*\Gamma(X, \HHl{X}(\OO_X))\simeq \bigoplus_{q-p=*}H^p(X, \Omega^{q, \log}_X).\]

\begin{definition}
    We define the log Chern character $\ch^{\log}(i_*E)$ of a perfect complex $E$ on $X$ as the image of $1\in H^0(X, \OO_X)$ under $\phi^{HH}_{i_*E}$ (using the HKR map):
    \[\phi^{HH}_{i_*E}(1)\in \HH_0^{\log}(X,D) \simeq \bigoplus_p H^p(X, \Omega^{p, \log}_X).\] 
\end{definition}

We have the following lemma.

\begin{lemma}\label{prop:chern1}
The logarithmic Chern character is normalized in the sense that
\[
\operatorname{ch}^{\log}(i_*\OO_X)=1.
\]
\end{lemma}

\begin{proof}


The strong log Fourier-Mukai kernel $i_*\OO_X$ is the unit for composition of strong log Fourier-Mukai kernels (see Lemma \ref{lem:unit}), hence preserves the unit class. Therefore
\[
\operatorname{ch}^{\log}(i_*\OO_X)=1
\]
implying our statement.
\end{proof}

\begin{example}\label{ex:chern}
    Consider the smooth log pair $(C, \pt)$, where $C$ is a smooth, connected, and proper curve. Let $L$ be a line bundle on $C$. We regard it as a strong log Fourier-Mukai kernel $i_*L\in D(C\times^{\log}C)$. Using the log HKR, the 0-th log Hochschild homology of $(C, \pt)$ is given by:
    \[\HH_0^{\log}(C, \pt) = H^0(C, \OO_C)=k\]
    as $H^1(C, \Omega^1_C(\pt))=0$. Therefore, $\phi^{HH}_{i_*L}$ provides a $k$-linear map $k\to k$. Consequently, the log Chern character of $i_*L$ has only its degree 0 component. 
    
    On the degree 0 component, this map is given by the evaluation and co-evaluation maps $\OO_{C}\to L\otimes L^\vee$ and $L\otimes L^\vee\to \OO_{C}$ which are the identity maps. Therefore $\ch^{\log}(i_*L)=1$ no matter what the line bundle is. This is in contrast with the usual Chern character.

    Similarly, consider a shifted line bundle $L[1]$, then its log Chern character has only its degree 0 component again. The degree 0 component is again given by the evaluation and co-evaluation maps, whose composite, in this case, is $-1$ (from the shift). Therefore,  $\ch^{\log}(L[1])=-1$. The same argument shows that $\ch^{\log}(L[n])=(-1)^n$ for any line bundle on $C$ and any integer $n$.
\end{example}


Furthermore, we have additivity for the log Chern character for direct sums. To show that, we need to prove an additivity result for the non-categorical unit and counit maps.

\begin{proposition}[Additivity of the unit map]\label{prop:unit_additivity}
Let $E_1, E_2 \in D(X\times^{\log}Y)$ be strong log Fourier--Mukai kernels, and set
\[
E \coloneq  E_1 \oplus E_2.
\]
Consider the unit-like map
\[
\eta_E: i_*\OO_X \longrightarrow p_*\big(E\boxtimes \tau(E^\vee \otimes q^*S_X^{\log})\big)
\]
constructed in Lemma~\ref{lem:right_adj}. Then under the canonical decomposition
\[
p_*\big(E\boxtimes \tau(E^\vee \otimes q^*S_X^{\log})\big)
\simeq
\bigoplus_{i,j=1}^2
p_*\big(E_i \boxtimes \tau(E_j^\vee \otimes q^*S_X^{\log})\big),
\]
the map $\eta_E$ factors through the diagonal summands and satisfies
\[
\eta_E = \eta_{E_1} \oplus \eta_{E_2},
\]
where $\eta_{E_i}$ denotes the unit map associated to $E_i$.
\end{proposition}


\begin{proposition}[Additivity of the counit map]\label{prop:counit_additivity}
Let $E_1, E_2 \in D(X\times^{\log}Y)$ be strong log Fourier--Mukai kernels, and set
\[
E \coloneq  E_1 \oplus E_2.
\]
Consider the counit-like map
\[
\epsilon_E:
p_*\big(E\boxtimes \tau(E^\vee \otimes t^*S_Y^{\log})\big)
\longrightarrow i_*\OO_X
\]
constructed in Lemma~\ref{lem:left_adj}. Then under the canonical decomposition
\[
p_*\big(E\boxtimes \tau(E^\vee \otimes t^*S_Y^{\log})\big)
\simeq
\bigoplus_{i,j=1}^2
p_*\big(E_i \boxtimes \tau(E_j^\vee \otimes t^*S_Y^{\log})\big),
\]
the map $\epsilon_E$ factors through the diagonal summands and satisfies
\[
\epsilon_E = \epsilon_{E_1} + \epsilon_{E_2},
\]
where $\epsilon_{E_i}$ denotes the counit map associated to $E_i$.
\end{proposition}

\begin{proof}[Proof of Propositions \ref{prop:counit_additivity} and \ref{prop:unit_additivity}]
    Both maps are induced by the co-evaluation and evaluation morphisms
\[
\OO \to E \otimes E^\vee
\quad \text{and} \quad
E^\vee \otimes E \to \OO.
\]
For $E = E_1 \oplus E_2$, we have decompositions
\[
E \otimes E^\vee \simeq \bigoplus_{i,j} E_i \otimes E_j^\vee,
\quad
E^\vee \otimes E \simeq \bigoplus_{i,j} E_i^\vee \otimes E_j.
\]
Under the identifications
\[
\RHom(\OO, E_i \otimes E_j^\vee) \simeq \RHom(E_j, E_i),
\quad
\RHom(E_i^\vee \otimes E_j, \OO) \simeq \RHom(E_j, E_i),
\]
the co-evaluation and evaluation maps correspond to identity morphisms $\mathrm{id}_{E_i}$.
Hence they factor through the diagonal summands $i=j$ and vanish on off-diagonal terms.
The result follows by functoriality of pullback, tensor product, and pushforward.
\end{proof}

As a corollary, we get that the log Chern character is additive for direct sums.

\begin{corollary}\label{cor:chernadditive}
    Let $E_1, E_2\in D(X\times^{\log}X)$ be two strong log Fourier-Mukai kernels. Then, $\ch^{\log}(E_1\oplus E_2)=\ch^{\log}(E_1)+\ch^{\log}(E_2)$.
\end{corollary}

We conjecture that a stronger statement also holds.

\begin{conjecture}
    Let 
    \[E_1\to E_2\to E_3\to E_1[1]\]
    be a distinguished triangle in $D(X\times^{\log}X)$ of strong log Fourier-Mukai kernels. Then,
    \[\ch^{\log}(E_1)+\ch^{\log}(E_3)=\ch^{\log}(E_2).\]
\end{conjecture}

If this conjecture holds, the log Chern character defines a morphism out of the $K$ theory $K^0(X)$ of perfect complexes. We also have a weak composition law.

\begin{proposition}[Composition property]
Let $E \in D(X\times^{\log}X)$ and $F \in D(X\times^{\log}X)$ be strong log Fourier--Mukai kernels. Then
\[
\ch^{\log}(F \circ E)
=
\phi^{HH}_E\big(\ch^{\log}(F)\big).
\]
\end{proposition}

\begin{proof}


This follows from the functoriality of Theorem \ref{thm:main}. In fact, 
\[\ch^{\log}(F\circ E)=\phi^{HH}_{F\circ E}(1)=\phi^{HH}_E\circ \phi^{HH}_F(1)\]
where the last expression can be identified with $\phi^{HH}_E(\ch^{\log}(F))$.
\end{proof}



In general, we conjecture that log Chern and Chern characters are compatible.
The set-up is the following. Let $(X,D_X)$ and $(Y, D_Y)$ be log smooth pairs, and let $E$ be a strong Fourier-Mukai kernel on $X\times^{\log}Y$. Consider the pushforward $\pi_*E$ of $E$ in $D(X\times Y)$ where $\pi$ is the blow-up map $\pi:X\times^{\log}Y\to X\times Y$. Since $\pi$ is proper, $\pi_*E$ is also a perfect complex, and we regard it as an ordinary Fourier-Mukai kernel on $D(X\times Y)$. 

We can consider the linear map induced by $E$ on log Hochschild homology $\phi^{HH}_E: \HH_*^{\log}(Y)\to\HH_*^{\log}(X)$ (see Theorem \ref{thm:main}). We can also consider the linear map induced by the Fourier-Mukai kernel, $\pi_*E$, on Hochschild homology $\phi^{HH}_{\pi_*E}:\HH_*(Y)\to \HH_*(X)$ (see \cite{caldararu2003mukai, kuznetsov2009}). 

We have the following conjecture.

\begin{conjecture}
     \label{conj:commutative}
     The following diagram commutes
    \[
    \begin{tikzcd}
        \HH_*(Y)\ar[r, "\phi^{HH}_{\pi_*E}"]\ar[d] & \HH_*(X)\ar[d]\\
        \HH^{\log}_*(Y)\ar[r, "\phi^{HH}_E"] & \HH^{\log}_*(X)
    \end{tikzcd}
    \]
    where the vertical maps are induced by the counit map $\Delta^*\Delta_*\Rightarrow i^*i_*$ from the self-intersection of the diagonal to the self-intersection of the log diagonal.
\end{conjecture}

\begin{remark}
    Example \ref{ex:chern} is an illustrative case of the conjecture above. In fact, in this case, $\HH_0(C)=H^0(C,\OO_C)\oplus H^1(C, \Omega^1_C)$, and the vertical arrow sends $H^1(C, \Omega^1_C)$ to 0.
\end{remark}

\begin{remark}
    If Conjecture \ref{conj:commutative} holds in the case of $X=Y$, then $\ch^{\log}(E)$ is the image of the ordinary Chern character $\ch(\pi_*E)$ under the natural map $c : \HH_*(X) \to \HH^{\log}_*(X)$. 
\end{remark}

\subsubsection{Log Chern character via expansion}
Consider the log scheme $(\PP^1, \pt)$. With $(X,D)$, they fit into the diagram
\[
\begin{tikzcd}
    & X\times^{\log}\PP^1\ar[dl, "\pi_2"]\ar[dr, "\pi_1"] &\\
    (\PP^1, \pt) & & (X,D)
\end{tikzcd}
\]




Now, we consider a strong log Fourier-Mukai kernel $E$ on  $X\times^{\log}\PP^1$. That gives rise to a map 
\[\phi^{HH}_E:R\Gamma(\PP^1, \HHl{(\PP^1, \pt)}(\OO_{\PP^1}))\to R\Gamma(X, \HHl{(X, D)}(\OO_X)).\]

Using the log HKR isomorphism, we get that 
\[ R\Gamma(\PP^1, \HHl{(\PP^1, \pt)}(\OO_{\PP^1}))=k[0].\]
Therefore, we get a distinguished class (namely $\phi^{HH}_E(1)$ in $\HH^{\log}_0(X)$). Using log HKR isomorphism again, we get a class in 
\[\bigoplus_{p}H^p(X, \Omega^{p, \log}_X)=\HH_0^{\log}(X)\]
that we call the log Chern character of the strong Fourier-Mukai kernel $E$, and denote it by $\ch_e^{\log}(E)$ where the subscript $e$ emphasizes that we consider strong log Fourier-Mukai kernels on $X\times^{\log}\PP^1$.

\begin{remark}
    We believe that this type of strong Fourier-Mukai kernel is closer in essence (than the previous subsection) to what log coherent sheaves should look like. 
\end{remark}

\begin{remark} 
    This framework suggests a way to define log Chern characters in the relative setting in a more natural way. Fix $S$, a log scheme, and consider the category of proper and log smooth log schemes over $S$. Given a proper, log smooth, log weakly separated (over $S$) log scheme $X\to S$, we can associate to $X/S$ the relative log Hochschild homology $R\Gamma(X,\HHl{X/S}(\OO_X))$. In the case of $S=X$, we get that the relative log Hochschild homology $R\Gamma(S,\HHl{S/S}(\OO_S))$ as a dg vector space is quasi-isomorphic to $R\Gamma(S, \OO_S)$ in degree 0 by the log version of HKR \cite{HABLICSEK2026127}. Thus, the diagram above would be replaced by the diagram
    \[
    \begin{tikzcd}
        & X\ar[ld]\ar[rd] &\\
        S & & X
    \end{tikzcd}    
    \]
    and the log Chern character would be the image of 1 under the map 
    \[\phi^{HH}_E:R\Gamma(S,\HHl{S/S}(\OO_S))\to R\Gamma(X,\HHl{X/S}(\OO_X)).\]
\end{remark}

Similarly to Section \ref{sec:chernperf}, we have that the log Chern characters are additive on direct sums. We omit the proof.

\begin{corollary}
    Let $E_1, E_2$ be two strong log Fourier-Mukai kernels on $X\times^{\log}\PP^1$. Then, $\ch_e^{\log}(E_1\oplus E_2)=\ch_e^{\log}(E_1)+\ch_e^{\log}(E_2)$.\qed
\end{corollary}

\subsection{Logarithmic Euler pairing}
\label{sec:logEuler}

In this section, we use the formalism of strong log Fourier-Mukai transforms to define the logarithmic Euler pairing between two strong log Fourier-Mukai kernels on $\PP^1\times^{\log}X$.

We start by discussing how to define the Euler pairing in the non-logarithmic case that we will tweak later to the logarithmic case.

Let $E$ and $F$ be perfect complexes on $X$. Regard them as Fourier-Mukai kernels for the Fourier-Mukai transforms $\phi_E:D(\pt)\to D(X)$ and $\phi_F:D(\pt)\to D(X)$. Then, the composite Fourier-Mukai transform $\phi_{F}^!\circ \phi_E$ can be regarded as a Fourier-Mukai transformation $D(\pt)\to D(\pt)$.

\begin{lemma}
    Let $E$ and $F$ be perfect complexes on $X$. Then, the Fourier-Mukai transformation $\phi_{F}^!\circ \phi_E$ acts as multiplication by the Euler pairing $\chi(F, E)$ on $\HH_*(\pt)$.
    
\end{lemma}

\begin{proof}
    By definition, we have that $\phi_{F}^!=F^\vee$. Thus, the kernel of the composite Fourier-Mukai transformation $\phi_{F}^!\circ \phi_E$ is quasi-isomorphic to $R\Gamma(X, F^\vee\otimes E)$ viewed as a kernel on $D(\pt\times \pt)=D(\pt)$. Since the action of $\phi_{F}^!\circ \phi_E$ on $\HH_0(\pt)$ is given by the co-evaluation and evaluation maps, it acts as multiplication by
    \[\chi(R\Gamma(X, F^\vee\otimes E)).\]
    Since $F$ is a perfect complex on $X$, this can be identified with
    \[\sum(-1)^i \dim R^i\Gamma(X, F^\vee\otimes E)=\sum(-1)^i \dim \Ext^i_X(F, E)=\chi(F,E)\]
    proving our claim.
\end{proof}

We are ready to define the logarithmic Euler pairing.

\begin{definition}
    Let $E,F\in D(\PP^1\times^{\log}X)$ be strong log Fourier-Mukai kernels. Consider the strong log Fourier-Mukai transformation $\phi_F^!\circ \phi_E$ and its action on log Hochschild homology
    \[(\phi_F^!\circ \phi_E)^{HH}:R\Gamma(\PP^1, \HHl{(\PP^1, \pt)}(\OO_{\PP^1}))\to R\Gamma(\PP^1, \HHl{(\PP^1, \pt)}(\OO_{\PP^1})).\]
    We define the log Euler pairing of $E$ and $F$, $\chi^{\log}(F, E)$, as
    \[(\phi_F^!\circ \phi_E)^{HH}(1)\in R\Gamma(\PP^1, \HHl{(\PP^1, \pt)}(\OO_{\PP^1})).\]
    Since $R\Gamma(\PP^1, \HHl{(\PP^1, \pt)}(\OO_{\PP^1}))=k[0]$, this gives us an element in $k$.
\end{definition}

\begin{remark}
    In other words, the log Euler pairing is the log Chern character of the strong log Fourier-Mukai kernel corresponding to the strong log Fourier-Mukai transform $\phi_F^!\circ \phi_E$.
\end{remark}

\begin{example}
    In this example, we consider $i_*\OO_{\PP^1}$ as a strong log Fourier-Mukai kernel corresponding to the strong log Fourier-Mukai transform $D(\PP^1)\to D(\PP^1)$ which is the identity (see Lemma \ref{lem:unit}). Furthermore, using the explicit formula for the non-categorical right adjoint (Lemma \ref{lem:right_adj}), we get that the non-categorical right adjoint of $i_*\OO_{\PP^1}$ is itself.

    Therefore, $\chi^{\log}(i_*\OO_{\PP^1}, i_*\OO_{\PP^1})=1$ from Proposition \ref{prop:chern1}. 
\end{example}

We conclude the paper with an illustrative computation of a log Euler pairing. The computation is quite involved, we only highlight the key steps.

\begin{example}
    Consider the log smooth pair $(X,D)=(\PP^2, H)$ where $H$ is a hyperplane and the embedding $f:(\PP^1, \pt)\to (\PP^2, H)$ where the image of $f$ intersects $H$ transversally. Let $\Gamma_f$ denote the graph of $f$ inside $\PP^1\times^{\log}\PP^2$. Its structure sheaf gives a strong log Fourier-Mukai kernel $\OO_{\Gamma_f}\in D(\PP^1\times^{\log}\PP^2)$. In this example, we compute $\chi^{\log}(\OO_{\Gamma_f}, \OO_{\Gamma_f})$.

    First, we compute the non-categorical right adjoint of $\OO_{\Gamma_f}$. Lemma \ref{lem:right_adj} tells us that $\phi^!_{\OO_{\Gamma_f}}$ is given by the kernel
    \[\tau((\OO_{\Gamma_f})^\vee\otimes (\OO_{\PP^1}(-1)\boxtimes \OO_{\PP^2})[1])=\tau(\Gamma_{f,*}(\OO_{\PP^1}\otimes S_{\PP^1/\PP^1\times^{\log}\PP^2})\otimes (\OO_{\PP^1}(-1)\boxtimes \OO_{\PP^2})[1])=\]
    \[=\tau\Gamma_{f,*}\left(S_{\PP^1/\PP^1\times^{\log}\PP^2}\otimes\Gamma_f^*(\OO_{\PP^1}(-1)\boxtimes \OO_{\PP^2}[1])\right)=\]
    \[=\tau\Gamma_{f,*}(\OO_{\PP^1}(3)\otimes \OO_{\PP^1}(-1)[-2]\otimes \OO_{\PP^1}(-1)[1])=\tau\Gamma_{f,*}\OO_{\PP^1}(1)[-1].\]
    Consider the diagram
    \[\begin{tikzcd}
     & & \PP^1\times^{\log}\PP^2\times^{\log}\PP^1 \ar[rd, "w"]\ar[ld, "v"]\ar[rr, "u"]& &\PP^1\times^{\log}\PP^1\\
     & \PP^1\times^{\log}\PP^2\ar[rd]\ar[ld] & &    \PP^2\times^{\log}\PP^1\ar[rd]\ar[ld] & \\
    \PP^1 &  & \PP^2 & & \PP^1.
\end{tikzcd}
\]
    Using the notation of the diagram, the composite morphism $\phi^!_{\OO_{\Gamma_f}}\circ \phi_{\OO_{\Gamma_f}}$ is given by the strong log Fourier-Mukai kernel (see Lemma \ref{lem:strongFMcomp})
    \[u_*(v^*\OO_{\Gamma_f}\otimes w^*\tau\Gamma_{f,*}\OO_{\PP^1}(1)[-1]).\]

    Similarly to Theorem \ref{thm:generalsquare2}, we have a commutative diagram
    \[
\begin{tikzcd}
    \PP^1\times^{\log} \PP^1\ar[r, "\alpha"]\ar[d, "\pi_1"] & \PP^1\times^{\log}\PP^2\times^{\log} \PP^1\ar[d, "v"]\\
    \PP^1\ar[r, "\Gamma_f"] & \PP^1\times^{\log} \PP^2.
\end{tikzcd}   
    \]
    where $\pi_1$ is the projection to the first component, $v$ is the projection on the first two components (as above) and $\alpha$ is given by $\Gamma_f$ on the first copy of $\PP^1$ and identity on the other copy. In other words, $\alpha$ is given by $(x,y)\mapsto (x, f(x), y)$ over $\PP^1\times\PP^1$.

    Similarly to Theorem \ref{thm:generalsquare2}, we have an equivalence of dg functors
    \[v^*\Gamma_{f,*}\simeq \alpha_*\pi_1^*\]
    showing that $v^*\OO_{\Gamma_f}=\alpha_*\OO_{\PP^1\times^{\log}\PP^1}$.

    The exact same computation applies to $w^*\tau\OO_{\Gamma_f}$ showing that 
    \[w^*\tau\OO_{\Gamma_f}=\beta_*\OO_{\PP^1\times^{\log}\PP^1}\]
    where $\beta:\PP^1\times^{\log}\PP^1\to \PP^1\times^{\log}\PP^2\times^{\log}\PP^1$ is given by identity on the first copy and by $\tau\Gamma_f$ on the last two copies.
    
    Therefore, $v^*\OO_{\Gamma_f}\otimes w^*\tau\Gamma_{f,*}\OO_{\PP^1}$ is given by the structure sheaf of the derived intersection
    \begin{equation}\label{dia:eulerpairingP^1}
    \begin{tikzcd}
            & \PP^1\times^{\log}\PP^1\ar[d, "\beta"]\\
    \PP^1\times^{\log}\PP^1\ar[r, "\alpha"] & \PP^1\times^{\log} \PP^2\times^{\log}\PP^1.
    \end{tikzcd}   
    \end{equation}

    The underived intersection is given by $\PP^1$ fitting into a Cartesian diagram
    \[
    \begin{tikzcd}
       \PP^1\ar[d, "i"]\ar[r, "i"]     & \PP^1\times^{\log}\PP^1\ar[d, "\beta"]\\
    \PP^1\times^{\log}\PP^1\ar[r, "\alpha"] & \PP^1\times^{\log} \PP^2\times^{\log}\PP^1.
    \end{tikzcd}   
    \]
    where $i$ is the log diagonal map $\PP^1\to \PP^1\times^{\log}\PP^1$.

   The tangent bundles restricted to $\PP^1$ of the spaces in this intersection problem look as follows: \[T_{\PP^1\times^{\log}\PP^1}|_{\PP^1}=T_{\PP^1}\oplus T_{\PP^1}^{\log}=\OO_{\PP^1}(2)\oplus \OO_{\PP^1}(1)\] and 
    \[T_{\PP^1\times^{\log}\PP^2\times^{\log}\PP^1}|_{\PP^1}=T_{\PP^1}\oplus T_{\PP^1}^{\log,\oplus 3}=\OO_{\PP^1}(2)\oplus \OO_{\PP^1}(1)^{\oplus 3}.\]
    The injection of vector bundles of $\PP^1$
    \[(T_{\PP^1\times^{\log}\PP^1}+T_{\PP^1\times^{\log}\PP^1})|_{\PP^1}\xrightarrow{t_\alpha+t_\beta}T_{\PP^1\times^{\log}\PP^2\times^{\log}\PP^1}|_{\PP^1} \]
    splits as it is given by the inclusion
    \[\OO_{\PP^1}(2)\oplus \OO_{\PP^1}(1)^{\oplus 2}\to \OO_{\PP^1}(2)\oplus \OO_{\PP^1}(1)^{\oplus 3}.\]
    Therefore, the derived intersection of Diagram \ref{dia:eulerpairingP^1} is formal (\cite{arinkin2019orbifold, grivaux2014formality}). The excess bundle $E$ is given by the cokernel of $\OO_{\PP^1}(2)\oplus \OO_{\PP^1}(1)^{\oplus 2}\to \OO_{\PP^1}(2)\oplus \OO_{\PP^1}(1)^{\oplus 3}$, namely,
    $\OO_{\PP^1}(1)$. Using the formality theorem of \cite{arinkin2019orbifold} or \cite{grivaux2014formality}, we have that the structure sheaf of the derived intersection (over $\PP^1$) is quasi-isomorphic to \[\Sym(E^\vee[1])=\Sym(\OO_{\PP^1}(-1)[1])=\OO_{\PP^1}\oplus \OO_{\PP^1}(-1)[1].\] 
    This means that
    \[v^*\OO_{\Gamma_f}\otimes w^*\tau \OO_{\Gamma_f}=\alpha_*i_*(\OO_{\PP^1}\oplus \OO_{\PP^1}(-1)[1]).\]
    Therefore,
    \[v^*\OO_{\Gamma_f}\otimes w^*\tau\Gamma_{f,*}\OO_{\PP^1}(1)[-1]=\alpha_*i_*(\OO_{\PP^1}(1)[-1]\oplus \OO_{\PP^1}).\]
    Since the map $u\circ \alpha$ is the identity map, we get that
    \[\phi^!_{\OO_{\Gamma_f}}\circ \phi_{\OO_{\Gamma_f}}=\phi_{i_*(\OO_{\PP^1}(1)[-1]\oplus \OO_{\PP^1})}.\]
    Therefore, 
    \[\chi^{\log}(\OO_{\Gamma_f}, \OO_{\Gamma_f})=\ch^{\log}(i_*(\OO_{\PP^1}(1)[-1]\oplus \OO_{\PP^1})).\]
    Since the log Chern character is additive for direct sums (see Corollary \ref{cor:chernadditive}), we have that
    \[\chi^{\log}(\OO_{\Gamma_f}, \OO_{\Gamma_f})=\ch^{\log}(i_*\OO_{\PP^1}(1)[-1])+\ch^{\log}(i_*\OO_{\PP^1}).\]
    By Example \ref{ex:chern}, this equals $-1+1=0$ giving us that
    \[\chi^{\log}(\OO_{\Gamma_f}, \OO_{\Gamma_f})=0.\]
    \end{example}

\section{Statements}
On behalf of all authors, the corresponding author states that there is no conflict of interest.

During this research, no data was generated.

\section{AI Declaration}
In the first version of this manuscript, the authors only used AI for editing and for grammar. All major ideas of the paper are the ideas of the authors. In the second version (the most current one), the authors used AI to a greater extent as a proofreader. AI was used not just for editing reasons, but for mathematical reasons as well. Many corrections have been made and new proofs emerged in the process. However, in this version as well, we retain the original major ideas. We itemize what AI has brought to our attention and fixed.

Lemma \ref{lem:flatprojection} was suggested to us by AI, and the proofs of Lemmas \ref{lem:gensquare} and \ref{lem:square2} were adapted accordingly. With AI, we could prove a stronger version of these lemmas which culminates in Theorem \ref{thm:generalsquare}. Furthermore, AI recommended to add Lemmas \ref{lem:adjunctioncoherence} and \ref{lem:adjcoherence} to make the proof of Theorem \ref{thm:main} more precise.


\bibliographystyle{alpha}
\bibliography{bib2}

\end{document}